\documentclass[12pt,a4paper]{amsart}
\makeatletter
\renewcommand\normalsize{%
    \@setfontsize\normalsize{11.7}{14pt plus .3pt minus .3pt}%
    \abovedisplayskip 10\p@ \@plus4\p@ \@minus4\p@
    \abovedisplayshortskip 6\p@ \@plus2\p@
    \belowdisplayshortskip 6\p@ \@plus2\p@
    \belowdisplayskip \abovedisplayskip}
\renewcommand\small{%
    \@setfontsize\small{9.5}{12\p@ plus .2\p@ minus .2\p@}%
    \abovedisplayskip 8.5\p@ \@plus4\p@ \@minus1\p@
    \belowdisplayskip \abovedisplayskip
    \abovedisplayshortskip \abovedisplayskip
    \belowdisplayshortskip \abovedisplayskip}
\renewcommand\footnotesize{%
    \@setfontsize\footnotesize{8.5}{9.25\p@ plus .1pt minus .1pt}%%
    \abovedisplayskip 6\p@ \@plus4\p@ \@minus1\p@
    \belowdisplayskip \abovedisplayskip
    \abovedisplayshortskip \abovedisplayskip
    \belowdisplayshortskip \abovedisplayskip}
\ifdefined\pdfpagewidth
\else
\fi
\calclayout
\makeatother

\usepackage{graphicx}
\usepackage{xcolor}
\usepackage{mathtools}
\usepackage{amsmath,amsthm,amsfonts,amssymb,latexsym,amscd,enumerate}
\usepackage{xy}
\xyoption{all}
\usepackage{comment}
\usepackage{stmaryrd}

\usepackage{palatino}
\usepackage[]{euler}
\usepackage{mathrsfs}
\DeclareMathAlphabet{\mathpzc}{OT1}{pzc}{m}{it}

\usepackage[hidelinks]{hyperref}
\usepackage[normalem]{ulem}

\numberwithin{equation}{subsection}
\swapnumbers

\newtheorem{theorem}[equation]{Theorem}
\newtheorem*{theorem*}{Theorem}
\newtheorem*{lemma*}{Lemma}

\newtheorem{proposition}[equation]{Proposition}
\newtheorem*{proposition*}{Proposition}
\newtheorem{lemma}[equation]{Lemma}
\newtheorem*{conjecture*}{Conjecture}

\theoremstyle{definition}
\newtheorem{definition}[equation]{Definition}

\newtheorem{example}[equation]{Example}

\newtheorem{remark}[equation]{Remark}
\newtheorem{remarks}[equation]{Remarks}

\newcommand{\R}{\mathbb{R}}
\newcommand{\C}{\mathbb{C}} 

\newcommand{\Z}{\mathbb{Z}}

\DeclareMathOperator{\Ad}{Ad} 
\DeclareMathOperator{\ad}{ad}

\DeclareMathOperator{\sign}{sign}

\newcommand{\lie}{\mathfrak}

\newcommand{\cH}{\mathcal{H}}

\newcommand{\bigG}{\mathbf{G}}
\newcommand{\bigH}{\mathbf{H}}

\newcommand{\ldoublebracket}{[\![}
\newcommand{\rdoublebracket}{]\!]}
\newcommand{\BigCell}{\mathcal{U}}

\usepackage{cancel} % for \cancel & \xcancel in math.

\newcommand{\isotropy}{\mathfrak{h}}
\newcommand{\HBundle}{\mathscr{H}}

\newcommand{\Oshima}{\mathscr{M}}
\newcommand{\Groupoid}{\mathscr{G}}
\newcommand{\Subgroupoid}{\mathscr{H}}
\newcommand{\Man}{\mathscr{M}}
\newcommand{\ManS}{\mathcal{S}}
\newcommand{\Satake}{\mathscr{X}}
\newcommand{\Closed}{\mathscr{C}}
\newcommand{\Subgp}{\mathscr{S}}

\begin{document}

\title[Lie groupoids and the tempered dual I]{Lie groupoids, the Satake compactification and the tempered dual,  I: The Satake groupoid}

\author{Jacob Bradd,  Nigel Higson and Robert Yuncken}

\dedicatory{Dedicated to Georges Skandalis}

\date{}

\thanks{
This research was begun during the 2025 thematic trimester programme on Representation Theory and Noncommutative Geometry at the Institut Henri Poincar\'e. 
The authors acknowledge support of the Institut Henri Poincaré (UAR 839 CNRS-Sorbonne Université), and LabEx CARMIN (ANR-10-LABX-59-01).  This article is also based upon work from COST Action CaLISTA CA21109, supported by COST (European Cooperation in Science and Technology): \url{www.cost.eu}.  
J.~Bradd and R.~Yuncken were supported by project OpART of the Agence Nationale de la Recherche (ANR-23-CE40-0016) and R.~Yuncken was also suppported by project CroCQG (ANR-25-CE40-5010).  N.~Higson was supported  by the NSF grant DMS-1952669.  
}

\begin{abstract}
The (maximal) Satake compactification associated to  a real reductive group $G$ is the closure of the  symmetric space of all maximal compact subgroups of $G$ within the compact space of all closed subgroups of $G$. We shall present three different views of a   groupoid that may be associated to the Satake compactification.
To begin, we shall define our Satake groupoid, as we shall call it, as a topological groupoid, and as a special case of a  general construction of Omar Mohsen.   Then we shall give a Lie-theoretic account of the Satake groupoid, borrowing from work  of Toshio Oshima.  Finally  we shall identify the Satake groupoid with the purely geometric b-groupoid of the Satake compactification, using the structure of the compactification as  a smooth manifold with corners.  In a subsequent paper we shall use the Satake groupoid to present a new proof of Harish-Chandra's principle, that all the tempered irreducible representations of $G$ may be constructed from discrete series representations using parabolic induction. 
\end{abstract}

\maketitle

\section{Introduction} 

The purpose of this paper and its sequel \cite{BraddHigsonYunckenOshimaPart2} is to examine  from the perspective of $C^*$-algebras and noncommutative geometry the following celebrated discovery of Harish-Chandra: 

\begin{theorem*}
Let $G$ be a real reductive group.  A tempered irreducible unitary representation of $G$ is either square-integrable, modulo center, or embeddable into a  representation that is unitarily parabolically induced from a  square-integrable, modulo center,   irreducible unitary representation of a Levi subgroup.
\end{theorem*}

Harish-Chandra's result played an important role in his  pursuit of the Plancherel formula.  In his review of  Harish-Chandra's Collected Works \cite{Langlands85}, Robert Langlands writes that 

\begin{quote}
     Harish-Chandra discovered quite early on the principles which allowed him to do this [obtain an explicit Plancherel formula] \dots  The critical notions are those of a Cartan subgroup, of a parabolic subgroup, of an induced, and of a square-integrable representation.
 \end{quote}

\smallskip

\begin{quote}
    \dots The first principle is that the representations [parabolically] induced from \dots  square-integrable [representations] suffice for the Plancherel formula \dots
\end{quote}

\smallskip

 \begin{quote}
      \dots The second is that \textup{[}a real reductive group\hspace{1.5pt}\textup{]} has square-integrable representations if and only if there are
      \textup{[}compact\hspace{1.0pt}\textup{]} Cartan subgroups \dots 
 \end{quote}
The second principle has long been  studied from a  geometric perspective, starting with Parthasarathy \cite{Parthasarathy72}, who described the discrete series (the square-integrable representations) using Dirac operators on the symmetric space associated to a real reductive group, continuing with the work of Atiyah and Schmid \cite{AtiyahSchmid77}, and culminating in the work of Lafforgue \cite{Lafforgue02}, who recovered Harish-Chandra's classification of the discrete series using noncommutative geometry and K-theory.   The theorem that we stated above is a precise version of the first principle.
Combined, the two principles paint  in broad outline a picture of the tempered dual of any real reductive group.  

Our approach to the theorem will proceed via  the (maximal) Satake compactification of the Riemmanian symmetric space associated to a real reductive group \cite{Satake60}. We shall incorporate the Satake compactification into an argument involving $C^*$-algebras by associating to it first a groupoid, and then the $C^*$-algebra of that groupoid.  The purpose of  this  paper is to describe the groupoid from three different points of view: those of topology, Lie theory and geometry.  As for $C^*$-algebras, their relevance to Harish-Chandra's principle is a reflection of Alain Connes' idea that the reduced $C^*$-algebra of a real reductive group should be viewed as the algebra of continuous functions on the tempered dual \cite{Connes:NCG}.  But we shall say no more about $C^*$-algebras here, and instead refer the reader to the sequel,  \cite{BraddHigsonYunckenOshimaPart2}.

The fastest way to present the \emph{Satake groupoid} (as we shall call it) is to use the following  simple observation of Omar Mohsen (which he has used to great effect in his own work \cite{Mohsen25CollegeDeFrance, Mohsen:blowup}): if $G$ is any group, and if $\{S\}$ is any collection of subgroups of $G$ that is closed under conjugation by elements of $G$, then the collection $\{ C\}$ of all cosets of all the subgroups in $\{ S\}$  carries the following structure of a groupoid over the object space $\{ S\}$: 
\[
\operatorname{source}(C) = C^{-1}C, \quad 
\operatorname{target}(C) = CC^{-1}
\quad \text{and} \quad 
C_1 \circ C_2 = C_1C_2.
\]
Now, the Satake compactification $\Satake$ of a real reductive group $G$ may be defined to be  the closure of the space of maximal compact subgroups of $G$ within the compact space of all closed subgroups of $G$.  So Mohsen's observation immediately applies, and   we obtain  a locally compact Hausdorff topological groupoid, as we shall explain  in Section~\ref{sec:Mohsen}.  This is our Satake groupoid $\Groupoid_{\Satake}$.

Although the above quickly characterizes the Satake compactification and the Satake groupoid, for computations it is much more convenient to construct both  compactification and  groupoid using Lie theory.  To do so we shall follow the approach of  Toshio Oshima \cite{Oshima78} to the Satake compactification. We shall present in Section~\ref{sec-oshima-space} a version of Oshima's work that we hope will be broadly accessible, and in Section~\ref{sec-oshima-groupoid} we shall describe the Satake groupoid from the Oshima point of view.

Oshima's construction makes it clear that the Satake compactification has finitely many $G$-orbits, which may be described using an Iwasawa decomposition $G{=}KAN$, as follows.  It is well-known in Lie theory that a  standard parabolic subgroup $P_I {=} M_I A_IN_I$ of $G$ may be asssociated to any subset $I$  of the set $\Sigma$ of simple restricted roots that is  obtained from the given Iwasawa decomposition, and that these are the only standard parabolic subgroups; see Section~\ref{sec-oshima-space} for more details. The $G$-orbits in $\Satake$ are also in bijection with the subsets $I\subseteq \Sigma$, with the orbit  $\Satake_I\subseteq \Satake$ being of  the type 
\[
\Satake_I \cong G \big / K_IA_I\overline N_I,
\]
where  $K_I{=}K\cap M_I$, and where $\overline{N_I}= \theta [ N_I]$. 
The Satake compactification, viewed as a collection of closed subgroups of $G$, consists of all the conjugates in $G$ of all the groups $H_I{=}K_I\overline N_I$. 

The orbit $\Satake_I$ is contained in the closure of the orbit $\Satake_J$ if and only if $I\subseteq J$.  It follows, for instance, that the orbit $\Satake_\Sigma$ is  open and dense in $\Satake$.  In addition, $K_\Sigma = K$, while the group  $A_{\Sigma}$ is the intersection of the center of $G$ with $A$, and $N_\Sigma$ is the trivial one-element group. The orbit  
\[
\Satake_\Sigma \cong G / KA_\Sigma ,
\]
therefore identifies, via the map $gKA_\Sigma\mapsto gKg^{-1}$, with the space of maximal compact subgroups of $G$.

As for the Satake groupoid,   each of the orbits $\Satake_I$ above is a locally closed, saturated subset of $\Satake$, and also a smooth embedded submanifold, and the reduction of the Satake groupoid $\Groupoid_{\Satake}$ to $\Satake_I$ has the form
\[
\Groupoid_{I} \cong G \big / K_I\overline N_I \underset{\,\,A_I} \times  G \big / K_I\overline N_I 
\]
(quotient by the diagonal right action of $A_I$).  For instance, the open and dense subgroupoid   $\Groupoid_{\Sigma}$ is   
\[
\Groupoid_{\Sigma} \cong G  / K  \underset{\,\,A_\Sigma} \times  G   / K .
\]
When $G$ has compact center, the group $A_\Sigma$ is trivial, and the above is simply the pair groupoid on $\Satake_\Sigma \cong G/K$.

Oshima actually constructed a smooth, closed $G$-manifold $\Oshima$ into which the variety of maximal compact subgroups of $G$ embeds as an open subset, while the Satake compactification embeds smoothly as a compact submanifold with corners.  We shall  in fact construct a Lie groupoid $\Groupoid_{\Oshima}$ over $\Oshima$ that we shall call the \emph{Oshima groupoid}. It is a quotient of the transformation groupoid for the action of $G$ on $\Oshima$. The Satake compactification,  viewed as a subset of $\Oshima$, is a saturated subset for the Oshima groupoid, and we shall  prove that the reduction of $\Groupoid_{\Oshima}$ to this subset is our Satake groupoid. 

Now, the  bounding submanifolds (of top dimension) of the Satake compactification within the Oshima space $\Oshima$ extend to smooth, closed hypersurfaces in $\Oshima$ that cross one another normally.  And to any  manifold, such as $\Oshima$, that is equipped with a finite family of normally crossing,  closed hypersufaces  there is  associated  a Lie groupoid  \cite{NistorWeinsteinXu99,Monthubert03}; the construction is an elaboration of ideas from the $b$-calculus of Richard Melrose \cite{MelroseICM90} (the Lie algebroid is Melrose's b-tangent bundle).  In the last section of this paper we shall present our third view of the Satake groupoid by identifing the Oshima group\-oid with this geometrically-defined \emph{$b$-groupoid}.

To conclude this introduction, it  is a pleasure to acknowledge with gratitude the kind assistance of Omar Mohsen, whose suggestion that his groupoid construction might be deployed in representation theory was the impetus for this work. It is a similar pleasure to acknowledge the interest of Alain Connes, who suggested to us a number of years ago that we ought to search for a groupoid that would explain Harish-Chandra's principle.

\section{The Satake compactification and the Satake groupoid}
\label{sec:Mohsen}

In this section, we shall present the topological  definition of the Satake compactification and groupoid of a real reductive group.

\subsection{Definition of the Satake compactification}
As mentioned in the introduction, the Satake compactification that we shall study is the closure of the set of all maximal compact subgroups of $G$  in  the space of all closed subgroups of $G$. The latter carries a natural compact Hausdorff topology; let us recall it.
    
\begin{definition}[\cite{Fell62}]
\label{def:Fell-topology}
Let $Z$ be a   topological space. We shall denote by  $\Closed(Z)$  the set of all \emph{nonempty} closed subsets of $Z$.  The \emph{Fell topology} on $\Closed(Z)$ is the topology generated by the base of open sets 
\begin{equation*}
  V\bigl (K ,\mathcal{U}\bigr) = \bigl \{\, X \in \Closed(Z) :  X\cap K {=} \emptyset \,  \text{ and } \, X\cap U \neq \emptyset \,\, \forall U\in \mathcal{U}\, \bigr \},
\end{equation*}
where $K$ ranges over all the compact subsets of $Z$,    and $\mathcal{U}$ ranges over all the  finite  families of open subsets of $Z$.
\end{definition}

\begin{theorem}[{\cite[Thm.\,1]{Fell62}}]
\label{thm-locally-compact-fell-topology}
If $Z$ is locally compact and second-countable, then the Fell topology on $\Closed(Z)$ is  locally  compact, Hausdorff and second-countable.
\end{theorem}

\begin{remark}
     In \cite{Fell62}, Fell includes the empty set as an element of his $\Closed(Z)$;  by doing so he obtains the one-point compactification of our $\Closed(Z)$.
\end{remark}

 The following definition and lemma offer  a full characterization of continuous maps into $\Closed(Z)$, analogous to the description of vector bundles via  classifying maps into Grassmannians. 

\begin{definition}
\label{def:bundle}
Let $B$ and $Z$ be topological spaces.  A \emph{continuous family of closed subsets of $Z$} over the space $B$ is  a closed subset 
$\mathbf{E}\subseteq B\times Z$ 
for which the  projection map $ \mathbf{E}\to B$ is an open mapping.
\end{definition}

\begin{lemma}
  \label{lem:bundle}
  Let $B$ and $Z$ be  locally compact topological spaces. 
  A  function 
  \(
  f :B \to  \Closed(Z)
  \)
  is continuous if and only if the subset 
  \[
  {\mathbf{E}} = \{\, (b,z)\in B{\times} Z :  z\in f(b)\, \}\subseteq B \times Z
  \]
is a continuous family of closed  subsets in the sense of  Definition~\textup{\ref{def:bundle}}. \qed
\end{lemma}

From now on we shall consider the Fell topology in the case of a second-countable, locally compact Hausdorff group $G$, since this is what is of interest to us (the second countability assumption is largely unnecessary, but it streamlines some arguments). 

\begin{lemma}[{\cite[Rmk.\,IV,\,p.474]{Fell62}}]
The subspace $\Subgp(G)\subseteq \Closed(G)$ of all closed subgroups  of $G$ is a compact metrizable subset of $\Closed(G)$.
\end{lemma}

 \begin{remark}
     A version of this result   was proved before \cite{Fell62}, by Chabauty \cite{Chabauty50}, using a different but equivalent definition of the topology on $\Subgp(G)$.
 \end{remark}

In the group case, the Fell topology has the following elementary  continuity property: 

\begin{lemma}
\label{lem:continuity}
    The maps
    \begin{align*}
      G \times \Closed(G) \to \Closed(G), \quad
      (g,X) \mapsto gX \\
      \Closed(G)\times G \to \Closed(G),
      \quad
      (X,g) \mapsto Xg
    \end{align*}
    are continuous. In particular, the adjoint action of $G$ on $\Subgp(G)$ is continuous. \qed
\end{lemma}

\begin{example}
\label{ex-mapping-homogeneous-space-to-chabauty-space}
It follows from the lemma that if $S$ is a closed subgroup of $G$, then the morphism 
\[
\begin{gathered}
G/S\longrightarrow \Subgp(G)
\\
gS \longmapsto gSg^{-1}
\end{gathered}
\]
is continuous. This also follows from Lemma~\ref{lem:bundle}.
\end{example}

Let us now specialize to the case of interest to us in this paper.

\begin{definition}
 Let $G$ be a real reductive group (in Section~\ref{sec-parabolic-subgroups} we shall say more about our precise assumptions on $G$; those details are not important in this section). 
 To be consistent with notation that we shall introduce in Section~\ref{sec-oshima-space},  we shall denote by  $\Satake_\Sigma$ the set of all maximal compact subgroups of $G$.
\end{definition}

Any two maximal compact subgroups of $G$ are conjugate,  so $\Satake_\Sigma$ is homogeneous space of $G$, and it therefore carries a natural topology, indeed a $G$-equivariant smooth manifold structure.  As we noted in Example~\ref{ex-mapping-homogeneous-space-to-chabauty-space}, the  inclusion map of $\Satake_\Sigma$ into $\Subgp(G)$
is continuous. It will emerge later that this map is a homeomorphism onto its image; see Section~\ref{sec-oshima-space}.   

\begin{definition}
\label{def:Satake}
The (maximal) \emph{Satake compactification}  of the space $\Satake_\Sigma$ of  maximal compact subgroups of $G$   is the closure 
\[
\Satake = \overline{\Satake_\Sigma}
\]
in the compact Hausdorff space  $\Subgp(G)$.
\end{definition}

This is not the  original definition of Satake in \cite{Satake60}, but it is equivalent; see  for instance  \cite[Thm.\,9.18]{GuivarchJiTaylor}.    It is not obvious at this point that $\Satake_\Sigma$ is an open subset of $\Satake$, but this is indeed so; see Section~\ref{sec-oshima-space} again.

\subsection{Definition of the Satake groupoid}
The Satake compactification is a conj\-ugation-invariant family of subgroups of $G$,  and  so we may apply Mohsen's groupoid construction that we described in the introduction.  We recall the construction.

In general, if $\Satake$ is any family of subgroups of a group $G$, closed under conjugation, then every left coset of a subgroup in $\Satake$ is also a right coset of a subgroup in $\Satake$.  
The \emph{coset groupoid} $\Groupoid_\Satake\rightrightarrows\Satake$ is the set $\Groupoid_\Satake$ of all cosets of all subgroups in $\Satake$ equipped with source and target maps
\[
   \operatorname{source}(gS) = S, \quad \operatorname{target}(Sg) = S,\qquad (g\in G,\ S\in\Satake),
\]
and with setwise product and inverse.

\begin{definition}
    \label{def:Satake_groupoid}
    Let $G$ be a real reductive group. The \emph{Satake groupoid} $\Groupoid_\Satake \rightrightarrows \Satake$ associated to $G$ is the coset groupoid of the Satake compactification.
\end{definition}

By definition, an element of $\Groupoid_{\Satake}$ is a closed subset of $G$.  Thus $\Groupoid_{\Satake}$ is a subset of $\Closed(G)$, and we shall be viewing it as such in this section.

The Satake groupoid shares its most basic features with the coset groupoid of any closed, conjugation-invariant subset of $\Subgp(G)$, for any locally compact $G$ (although for simplicity we shall focus on second-countable groups).  
 
\begin{lemma}
\label{lem-coset-groupoid-is-closed-subset}
 Let $G$ be a second-countable, locally compact Hausdorff topological group, and let  $\Satake$ be  a closed, conjugation-invariant subset of the space $\Subgp(G)$ of all closed subgroups of $G$.  
The coset groupoid   $\Groupoid_{\Satake}$ is a closed subset of $\Closed(G)$.  

More precisely, if $(C_n)$ is a sequence of cosets in $\Groupoid_\Satake$ which converges to some closed subset $C\in\Closed(G)$, then we can find some convergent sequences $g_n\to g \in G$ and $S_n\to S\in\Subgp(G)$ such that $C_n=g_nS_n$ for all $n$ and $C=gS$.
\end{lemma}

\begin{proof}
Suppose that some sequence  $(C_n)$ in $\Groupoid_{\Satake}$ converges to $C\in\Closed(G)$.  Pick $g\in C$ and let $V_1 \supset V_2 \supset \cdots$ be a nested neighborhood basis for $g\in G$. Since $C \in V(\emptyset ,\{V_k\})$  (notation from Definition~\ref{def:Fell-topology}), it follows that for every $k$ there is an $N_k$, without loss of generality greater than $N_{k-1}$, such that $C_n \in V(\emptyset ,\{V_k\})$ for all $n\geq N_k$.  Now, for every $N_k\leq n < N_{k+1}$, pick $g_n \in C_n \cap V_k$.  Then $g_n\to g$.  Since $g_n\in C_n$ we can write $C_n = g_nS_n$ for some $S_n\in\Satake$, and we get, thanks to Lemma \ref{lem:continuity},
\[
  \lim_{n\to \infty} S_n = \lim_{n\to\infty} g^{-1}g_nS_n = \lim_{n\to\infty} g^{-1}C_n = g^{-1}C.
\]
Since $\Satake$ is closed, we deduce that $g^{-1}C = S$ for some closed subgroup $S\in \Satake$.  The result follows.
\end{proof}

\begin{theorem}
\label{thm-topological-structure-on-coset-groupoid}
Let $G$ be a second countable, locally compact Hausdorff topological group, and let  $\Satake$ be any closed, conjugation-invariant subset of  $\Subgp(G)$.  
When equipped with the Fell topology it inherits from $\Closed(G)$, the coset groupoid $\Groupoid_{\Satake}$ is a locally compact Hausdorff topological groupoid with open range and source maps. 
\end{theorem}

\begin{proof}
The set-wise inverse $C\mapsto C^{-1}$ is Fell-continuous on $\Closed (G)$ because the inverse operation is a homeomorphism from $G$ to itself.

Every  sequence of composable pairs of morphisms in $\Groupoid_{\Satake}$ has the form
$ ((g_n S_n ,S_n  h _n))_{n=1}^\infty$, 
with $g_n,h_n\in G$ and $S_n\in \Satake$.
If this sequence is convergent, then by Lemma \ref{lem-coset-groupoid-is-closed-subset} we may assume that $g_n \to g$, $h_n \to h$ and $S_n \to S$, for some $g,h\in G$ and some $S\in\Satake$.  Therefore 
\[
\lim_{n\to \infty} g_n S_n \circ S_n  h _n
=
\lim_{n\to \infty} g_nS_nh_n
= gSh
= \lim_{n\to \infty} g_n S_n \circ \lim_{n\to \infty} S_n  h _n,
\]
so that composition is continuous. Since the  source and target maps, viewed as maps to the unit elements in $\Groupoid_{\Satake}$, are compositions of composition and inverse maps, they are continuous, too.  

It remains to prove that the source and target maps are open. Since the target map is the composition of the inverse map with the source map, it suffices to prove the source map is open.  
For this, we need to prove that if  $V(K,\mathcal{U} )$ is a basic open set, as in Definition~\ref{def:Fell-topology} (where $K$ is any compact subset of $G$, and not necessarily a subgroup), then the image of the set $\Groupoid_\Satake \cap V(K,\mathcal{U})$
under the source map is an open set in $\Satake$.

Let $gS\in V(K,\mathcal{U} )$. Then $ S \in V(g^{-1}K,g^{-1}\mathcal{U})$, and moreover 
\[
S' \in \Satake \cap V(g^{-1}K,g^{-1}\mathcal{U})
\quad \Rightarrow \quad gS' \in V(K,\mathcal{U}),
\]
This shows that 
\[
\Satake \cap V(g^{-1}K, g^{-1}\mathcal{U})
\subseteq \operatorname{source}\bigl [ \Groupoid_\Satake \cap V(K,\mathcal{U})\bigr ] ,
\]
and therefore that $S$ lies in the interior of the image of $\Groupoid_\Satake \cap V(K,\mathcal{U})$ under the source map, as required.
\end{proof}

\section{The Oshima space} 
\label{sec-oshima-space}

In this section we shall  present a small   variation of Oshima's construction of a smooth, closed $G$-manifold that includes the Satake compactification as a full-dimensional compact submanifold with corners. Oshima's manifold is in fact the union of a finite set of  copies of the Satake compactification, joined together along  various common boundary faces.

\subsection{Background information on reductive groups and parabolic subgroups} 
\label{sec-parabolic-subgroups}

We shall work with the real reductive groups as defined by Knapp in \cite[Sec.VII.2]{KnappBeyond}. This includes Harish-Chandra's class \cite[Sec.\,3]{HarishChandra75}, and agrees with it   for linear groups \cite[pp.\,447-449]{KnappBeyond}.

Let $G$ be a real reductive group, as above, and fix  the following: 
\begin{enumerate}[\rm (i)]

\item A Cartan involution $\theta \colon G \to G$ (we shall use the same symbol for the involution on the Lie algebra). 

\item A nondegenerate, $\theta$-invariant symmetric bilinear form on $\lie{g}$ that is negative-definite  on the $+1$  eigenspace of $\theta$ and positive-definite on the $-1$ eigen\-space.

\item An Iwasawa decomposition $G {=} KAN$ \cite[Prop.\,7.31]{KnappBeyond}  that is compatible with $\theta$, and its associated  minimal parabolic subgroup $P_{\min}=MAN$, where $M$ is the centralizer of $\mathfrak{a}$ in $K$.

\end{enumerate}
Recall that the \emph{standard parabolic subalgebras} of $\mathfrak{g}$ for this choice of minimal parabolic subgroup are the Lie subalgebras of $\lie{g}$ that contain $\lie{p}_{\min}$.  They are in bijection with the subsets of the set $\Sigma\subseteq \Delta^+(\lie{g}, \lie{a})$ of simple restricted roots in the following way: given $I\subseteq \Sigma$ form
\begin{align}
\label{eq-def-a-i}
\lie{a}_I 
    & = \bigcap \bigl \{ \, \operatorname{kernel}(\alpha) : \alpha \in I\,\bigr\},
    \\
\label{eq-def-m-i}
\lie{m}_I 
    & = \lie{m}\oplus \lie{a}_I^\perp \oplus \bigoplus_{\substack{\gamma\in\Delta(\mathfrak{g},\mathfrak{a})\\ \gamma[\lie{a}_I] =0}} \lie{g}_\gamma ,
    \\
\label{eq-def-n-i}
\lie{n}_I 
    & =  \bigoplus_{\substack{\gamma\in\Delta^+(\mathfrak{g},\mathfrak{a})\\ \gamma[\lie{a}_I] \ne 0}} \lie{g}_\gamma ,
\end{align}
with $\mathfrak{g}_\gamma$ the restricted roots spaces of $(\mathfrak{g},\mathfrak{a})$  (see for instance \cite[p.\,455]{KnappBeyond}) and $\lie{a}_I^\perp$ the orthogonal complement of $\lie{a}_I$ in $\lie{a}$ using the form in (ii) above.  All of the above are Lie  subalgebras of $\lie{g}$;  the algebras  $\lie{m}_I$ and $\lie{a}_I$   commute with one another;   and both normalize $\lie{n}_I$.
 The corresponding standard parabolic subalgebra is  
\[
\lie{p}_I = \lie{m}_I \oplus \lie{a}_I\oplus \lie{n}_I.
\]
For all the above, see \cite[Sec.\,VII.7]{KnappBeyond}.

Now form the Lie groups
\[
\begin{aligned}
L_I 
    & = \text{Centralizer of $\lie{a}_I$ in $G$}
    \\
N_I 
    & = \text{Connected Lie subgroup of $G$ with Lie algebra $\lie{n}_I$} .
\end{aligned}
\]
These are closed subgroups of $G$. Moreover  $L_I$ normalizes $N_I$. The product 
\[
P_I = L_IN_I
\]
is also a closed subgroup of $G$; it is  the \emph{standard parabolic subgroup} associated to $I$.  The group $L_I$  decomposes as a Cartesian product of Lie groups 
\begin{equation}
\label{eq-split-decomp-of-l-i}
L_I = M_I A_I ,
\end{equation}
where $A_I$ is the connected Lie subgroup of $G$ with Lie algebra $\lie{a}_I$, and $M_I$ is the mutual kernel of all Lie group morphisms from $L_I$ to $\R^\times _+$; its Lie algebra is $\lie{m}_I$.  See \cite[Sec.\,VII.7]{KnappBeyond} again.

\begin{theorem}[See {\cite[Thm.\,3]{Moore64}} or {\cite[Prop.\,7.83]{KnappBeyond}}]
\label{thm-normalizer-thm-of-moore}
If $I$ is any subset of $\Sigma$, then   
$
P_I =\operatorname{Normalizer}_G(\mathfrak{p}_I) = \operatorname{Normalizer}_G(\mathfrak{n}_I)$.
\end{theorem}

Because of the Bruhat decomposition,  $G {=} P_{\min}{\cdot} \operatorname{Normalizer}_K(\mathfrak{a}){\cdot}P_{\min}$, the proof is  reduced  to checking that if $w \in \operatorname{Normalizer}_K(\mathfrak{a})$, and if $w$ conjugates the $P_I$ to itself, or $N_I$ to itself, then $w\in K_I = K \cap M_I$.  Since $w$ normalizes $\mathfrak{a}$, it permutes the restricted roots of $G$, and the proof may be completed by a combinatorial argument.
 
 \begin{theorem}[{\cite[Lemma\,6.2]{HarishChandra75}}]
 \label{thm-tits-theorem}
If $I, J \subseteq \Sigma$, and if $P_I$ and $P_J$ are conjugate in $G$, or $N_I$ and $N_J$ are conjugate in $G$, then $I=J$.
\end{theorem}

This is a result of Tits \cite{Tits62}, and the proof is similar: thanks to the Bruhat decomposition, it suffices to prove that if $I{\ne}J$, then no element of $N_K(\lie{a})$ can conjugate $P_I$ to $P_J$, or $N_I$ into $N_J$.  See for instance \cite[Prop.\,VI.20]{BourbakiLieGroups4to6} for an exposition of the combinatorial proof of this latter fact in the language of root systems (Tits uses $BN$-pairs). 

\subsection{A smooth family of Lie subgroups}
\label{sec-smooth-family-of-subgroups}
In this section we shall describe the main step in Oshima's construction. 

\begin{definition}
If $\Sigma$ is any finite set, then we shall denote by 
$\R^\Sigma$ the space of finite sequences   $(t_\alpha)_{\alpha\in \Sigma}$  of real numbers that are indexed by the elements of $\Sigma$ (or in other words, functions from $\Sigma$ to $\R$). In addition we shall denote by  $\R^\Sigma_+$ the subset of $\R^\Sigma$ comprised of sequences that take only strictly positive values, and denote by $\R^\Sigma_\times$ the subset of $\R^\Sigma$ comprised of sequences that take only nonzero values.  
\end{definition}

\begin{definition}
\label{def-coordinatizing-a}
Let $G=KAN$ be a real reductive group, and let $\Sigma$ be the set of simple roots associated to the given Iwasawa decomposition.
\begin{enumerate}[\rm (i)]

\item 
For $a\in A$ and $\alpha \in \Sigma$, let
$a^\alpha = e^{\alpha(\log(a))}$.

\item 
For  $t\in \R^\Sigma_+$, denote by $a_t\in A$ any element for which 
$a_t^\alpha = t_\alpha$  for all $\alpha \in \Sigma$.
\end{enumerate}
\end{definition}

When $G$ has compact center, the element $a_t$ in (ii) above is unique.  
In general, when $G$ does not have compact center, the ratio of any two choices lies in the center of $G$. Because of this, we may unambiguously define a subgroup $H_t \subseteq G$, as follows:

 \begin{definition}
\label{def-h-t-t-non-zero}    
Let $G=KAN$ be a real reductive group, and let $\Sigma$ be the set of simple roots associated to the given Iwasawa decomposition.
\begin{enumerate}[\rm (i)]

\item 
For  $t\in \R^\Sigma_+$,  we shall write   $H_t = \Ad_{a_t} [K ]$.  

\item 
For $t\in \R^\Sigma_{\times}$, we  define $|t|$ to be the sequence of absolute values $\{ \, |t_\alpha|\,\}_{\alpha\in \Sigma}$, and then write  
$H_t =  H_{|t|}$.
\end{enumerate}
\end{definition}

We aim to prove the following result:

 \begin{theorem}
\label{thm:H}
Let us write $\bigG = G \times\R^\Sigma$.
The closure of the subset 
\[
\{\, (h,t) :h \in H_t,\,\,\,  t \in \R^\Sigma_\times \,\} \subseteq \bigG  
\]
is a smooth embedded  submanifold $\bigH \subseteq \bigG$, and the projection map 
\[
\bigH \longrightarrow \R^\Sigma
\]
is a smooth submersion whose fibers are  closed Lie subgroups $H_t \subseteq G$.   
\end{theorem}

\begin{definition}
    \label{def_h}
   Let $\gamma\in \Delta^+(\mathfrak{g},\mathfrak{a})$. According to elementary root system theory, we may express $\gamma$ in a unique way as a combination
\[
 \gamma = \sum_{\alpha\in\Sigma} n_{\gamma,\alpha} \alpha,
 \]
 where each $n_{\gamma\alpha}$ is a nonnegative integer. For $t\in \R^\Sigma $, we shall write 
\begin{equation}
\label{eq:t2gamma}
  t^{2\gamma} = \prod_{\alpha\in\Sigma} t_{\alpha}^{2n_{\gamma,\alpha}} .
\end{equation}
(To be clear we shall write 
\[
t_{\alpha}^{2n_{\gamma,\alpha}}
= 
\begin{cases} 
1 & \text{if $t_\alpha=0$ and $ n_{\gamma,\alpha}  = 0$}
\\
0 & \text{if $t_\alpha=0$ and $n_{\gamma,\alpha} > 0$.}
\end{cases}
\]
With this convention,    $t^{2\gamma}$ becomes a polynomial function of $t$.)
\end{definition}

\begin{definition}
\label{def-h-t}
    For $t\in \R^\Sigma$,  define a linear subspace $\mathfrak{h}_t\subseteq \mathfrak{g}$ by
    \[
      \isotropy_t = \lie{m} \oplus \bigoplus_{\gamma\in\Delta^+(\mathfrak{g},\mathfrak{a})} \bigl \{ \,  t^{2 \gamma}X {+} \theta(X) : X_\gamma\in \lie{g}_\gamma\, \bigr \} 
    ,
    \]     
    where $\lie{m}$ is the centralizer of $\mathfrak{a}$ in $\mathfrak{k}$, and where $\theta$ is the Cartan involution.  
\end{definition}

\begin{proposition}
\label{prop:h}
Each subspace $\isotropy_{t}\subseteq \mathfrak{g}$ is a Lie subalgebra of $\mathfrak{g}$.  If $t\in \R^\Sigma_\times$, then $\lie{h}_t$ is the Lie algebra of the group $H_t$ in Definition~\textup{\ref{def-h-t}}. 
\end{proposition}

This is a consequence of the following two computations.

\begin{definition} 
\label{def-a-action}
Define a smooth action  of the Lie group $A$ on the linear space $\R^\Sigma$ by
\begin{equation*}
  (a\cdot t)_\alpha = a^{\alpha} t_\alpha\qquad \forall \alpha \in \Sigma .
\end{equation*}
\end{definition}

\begin{lemma}
If $a\in A$ and $t\in \R^\Sigma$, then
$(a {\cdot} t)^{2\gamma} = a^{2\gamma} t^{2\gamma}$. \qed
\end{lemma}

\begin{lemma}
\label{lem-h-conjugation}
If $a\in A$ and $t\in \R^\Sigma$, then  $\isotropy_{a\cdot t} = \Ad_a\isotropy_t$. \qed
\end{lemma}

\begin{proof} 
 We calculate that
    \begin{equation*}
    \begin{aligned}
      \Ad_a \bigl \{\, t^{2\alpha}X{+}\theta(X): X  \in \mathfrak{g}_\gamma \,\bigr \}
      & = \bigl \{ \, a^\alpha t^{2\alpha}X + a^{-\alpha}\theta (X) : X\in \mathfrak{g}_\gamma \,  \bigr \} 
      \\
      & = 
       a^{-\alpha} \cdot \bigl \{\, (a \cdot t)^{2\alpha}X +  \theta (X) : X \in \mathfrak{g}_\gamma \,\bigr \} 
      \\
      & = \bigl \{ \, (a \cdot t)^{2\alpha}X +  \theta (X) : X \in \mathfrak{g}_\gamma \,\bigr\},
      \end{aligned}
    \end{equation*}
    which, in view of the formulas  for $\mathfrak{h}_t$ and $\mathfrak{h_{a\cdot t}}$ in Definition~\ref{def-h-t},  proves the lemma. 
\end{proof}

\begin{proof}[Proof of Proposition~\ref{prop:h}]  
    If  $t_\alpha{=}\pm 1$ for all $\alpha$, then $t^{2 \gamma} =1$ for all $\gamma$, and therefore $\mathfrak{h}_t = \mathfrak{k}$. If $t_\alpha {\ne} 0$ for all $\alpha$, then   it  follows from the Lemma~\ref{lem-h-conjugation}  that $\mathfrak{h}_t$ is the Lie algebra of $H_t$, and so in particular is a Lie algebra.  For general $t $, it follows from a limiting argument that $\lie{h}_t$ is a Lie algebra, since it is clear from  Definition~\ref{def-h-t} that the $\mathfrak{h}_t$ are the fibers of a smooth vector subbundle of the constant bundle over $\R^\Sigma$ with fiber $\mathfrak{g}$. 
\end{proof}

We can describe the Lie algebras $\mathfrak{h_t}$ more explicitly as follows.  We have a direct sum decomposition
\begin{equation}
\label{eq-direct-sum-decomp-of-h-t}
  \isotropy_{t}  = \lie{k}_t \oplus \theta\bigl [ \lie{n}_t\bigr ],
\end{equation}
where 
\[
\lie{k}_t =  \lie{m} \oplus \bigoplus_{\substack{\gamma\in\Delta^+(\mathfrak{g},\mathfrak{a})\\ t^{2\gamma} \ne 0}} \bigl \{ \,   t^{2 \gamma} X_\gamma {+} \theta(X_\gamma) : X_\gamma\in \lie{g}_\gamma\, \bigr \} 
\quad \text{and} \quad    \lie{n}_t = \bigoplus_{\substack{\gamma\in\Delta^+(\mathfrak{g},\mathfrak{a})\\ t^{2\gamma} =0}}  \mathfrak{g}_\gamma    .
\]
If, for a given $t\in \R^\Sigma$, we set  
\begin{equation}
\label{eq:I}
I = \{ \, \alpha \in \Sigma: t_\alpha \ne 0\,\},
\end{equation}
then  
\[
t^{2\gamma}\ne 0 
\quad \Leftrightarrow \quad  
n_{\gamma,\alpha} = 0\,\,  \forall \alpha \notin I \quad \Leftrightarrow \quad  \gamma [ \mathfrak{a}_I] = 0,
\]
where $\mathfrak{a}_I$ is the mutual kernel of all $\alpha \in I$, as in \eqref{eq-def-a-i}. So we may also write 
\[
\lie{k}_t =  \lie{m} \oplus \bigoplus_{\substack{\gamma\in\Delta^+(\mathfrak{g},\mathfrak{a})\\ \gamma [ \mathfrak{a}_I] =0}} \bigl \{ \,   t^{2 \gamma} X_\gamma {+} \theta(X_\gamma) : X_\gamma\in \lie{g}_\gamma\, \bigr \} 
\quad \text{and} \quad    
\lie{n}_t = \bigoplus_{\substack{\gamma\in\Delta^+(\mathfrak{g},\mathfrak{a})\\ \gamma [ \mathfrak{a}_I] \ne 0}} \mathfrak{g}_\gamma    .
\]

\begin{proposition}
    \label{prop:h_t_structure}
For every $t\in \R^\Sigma$, the Lie subalgebra $\lie{h}_t$ is conjugate, via an element of $A$,  to a Lie subalgebra of the form
\begin{equation*}
\mathfrak{h}_I = \mathfrak{k}_I \oplus \overline{\mathfrak{n}}_I\, 
\end{equation*}
where:
\begin{enumerate}[\rm (i)]

\item 
$I = \{ \, \alpha \in \Sigma: t_\alpha \ne 0\,\}$,  as in \eqref{eq:I}, 

\item 
$\overline{n}_I = \theta [\mathfrak{n}_I]$, with $\mathfrak{n}_I$ as in \eqref{eq-def-n-i}, and 

\item 
$\lie{k}_I= \lie{k}\cap \lie{m}_I $, with $\lie{m}_I$  as in $\eqref{eq-def-m-i}$.
\end{enumerate}
The Lie algebra $\lie{k}_t$ normalizes, $\lie{n_t}$, so that the displayed decomposition is a semidirect product decomposition. 
\end{proposition}

\begin{proof}
According to the definitions in the statement of the proposition, 
\begin{equation}
\label{eq-def-h-i-2}
\lie{k}_I =\lie{m} \oplus \bigoplus_{\substack{\gamma\in\Delta^+(\mathfrak{g},\mathfrak{a})\\ \gamma[\mathfrak{a}_I]=0}} \bigl \{ \,   X_\gamma {+} \theta(X_\gamma) : X_\gamma\in \lie{g}_\gamma\, \bigr \} 
\quad \text{and} \quad 
 \lie{n}_I  =\bigoplus_{\substack{\gamma\in\Delta^+(\mathfrak{g},\mathfrak{a})\\  \gamma[\mathfrak{a}_I]\ne 0 }}
 \mathfrak{g}_\gamma   ,
\end{equation}
so that $\lie{h}_I = \lie {h}_{t'}$ for any $t'\in \R^\Sigma$ with 
\[
t'_{\alpha} =
\begin{cases}
    \pm 1 & \alpha \in I
    \\
    0 & \alpha \notin I.
\end{cases}
\]
But the spaces $\lie{k}_t$ and $\lie{n_t} $ satisfy 
\begin{equation}
    \label{eq-conjugates-of-k-t-and-n-t}
\Ad_a [ \lie{k}_t] = \lie{k}_{a\cdot t} 
\quad \text{and} \quad 
\Ad_a [ \lie{n}_t] = \lie{n}_{t}.
\end{equation}
So $\mathfrak{h}_t$ is conjugate to some $\mathfrak{h}_{t'}$, as above.
\end{proof}

Let us turn now from Lie subalgebras to Lie subgroups.

\begin{definition}
    \label{def:H-circ}
For each $t\in\R^\Sigma$, denote by  $H_{0,t}$ the connected Lie subgroup of $G$ with Lie algebra $\mathfrak{h}_t$.  For $I\subseteq \Sigma$ denote by  $H_{0,I}$ the connected Lie subgroup of $G$ with Lie algebra $\mathfrak{h}_I$.
\end{definition}

\begin{lemma} 
\label{lem-h-0-t-is-closed}
Each $H_{0,t}$ and each $H_{0,I}$ is a closed subgroup of $G$.
\end{lemma} 

\begin{proof} 
By Proposition \ref{prop:h_t_structure}, it suffices to prove that $H_{0,I}$ is closed,  for every $I\subseteq \Sigma$. The connected Lie subgroup $K_{0,I}$ with Lie algebra $\mathfrak{k}_I$ is the   connected component of the identity in the centralizer of $\mathfrak{a}_I$ in $K$, and is therefore closed, and indeed compact.  The connected Lie subgroup with Lie algebra  $\overline{ \mathfrak{n}}_I$ is a connected Lie subgroup of the exponential group $\overline{N}$, and is therefore closed in $\overline{N}$ and in $G$.  The group $H_{0,I}$ is the product of these two subgroups, and is therefore closed.
\end{proof}

\begin{lemma} If $t\in \R^\Sigma_\times $, then  the Lie subgroup $H_t\subseteq G $ in Definition~\textup{\ref{def-h-t-t-non-zero}} is equal to the product
$
H_t = MH_{0,t}$, 
where, as before,  $M$ is the centralizer of $\mathfrak{a}$ in $K$. 
\end{lemma} 

\begin{proof} 
The group $M$  is included in $K$ and  meets each connected component of $K$; see \cite[Lemma\,4.11]{HarishChandra75} or \cite[Prop.\,7.33]{KnappBeyond}.  Moreover it commutes with  $a_t$, and therefore
\[
H_t = \Ad_{a_t} [K] = M \cdot \Ad_{a_t} [K_0] = M H_{0,t} ,
\]
as required.
\end{proof} 

 This prompts the following definition:
 
\begin{definition}
\label{def-of-all-h-t-groups}
If $t$ is any element of $\R^\Sigma$, then  we shall write 
$  H_t = M H_{0,t}$. In addition, for $I\subseteq \Sigma$ we shall write $H_I = MH_{0,I}$.
\end{definition} 

The group $H_I$ is a closed subgroup of $G$, and it decomposes as a semidirect product  
\begin{equation*}
    H_I = K_I\ltimes \overline{N}_I,
\end{equation*}
where $K_I = M K_{0,I}$ (and where $K_{0,I}$ is the compact, connected Lie subgroup of $G$ with Lie algebra $\mathfrak{k}_I$, as in the proof of Lemma~\ref{lem-h-0-t-is-closed}).

\begin{theorem} 
\label{thm-normalizer-of-h-i}
The normalizer in $G$ of the Lie subalgebra  $\lie{h}_I{\subseteq} \lie{g}$  is $A_IH_I$.  Moreover, $\lie{h}_I$ is conjugate to $\lie{h}_{J}$ if and only if $I = J$.
\end{theorem}

We shall prove this using the following Lie algebra computation: 

\begin{lemma}
\label{lem-nilpotent-computation}
If $I$ is any subset of $\Sigma$, then 
\[
 \{\, X\in \lie{h}_I : \text{\rm $\ad_X\colon \lie{g}\to \lie{g}$ is nilpotent}\,\} =(\lie{k}_I \cap \lie{z}(\lie{g})) \oplus \overline{\lie{n}}_I,
\]
where $\lie{z}(\lie{g})$ is the center of $\lie{g}$.
\end{lemma} 

\begin{proof}
Let  $X =Y {+} Z \in \lie{k}_I \oplus \overline{\lie{n}}_I$ and suppose  that $\ad_X\colon \lie{g}\to \lie{g}$ is nilpotent, say $\ad_X^k{=}0$. 
If we write $\lie{g}$ as a direct sum of $\lie{a}_I$-root spaces (meaning eigenspaces for the adjoint action of $\lie{a}_I$), then the diagonal part of $\ad_X^k$  for this direct sum decomposition is $\ad_Y^k$.  So $\ad_Y^k{=}0$.  But $\ad_Y$ acts as a skew-symmetric operator on $\lie{g}$ for the inner product $\langle \underbar{\,\,}\,,\underbar{\,\,} \rangle = B(\underbar{\,\,}\,,\theta(\underbar{\,\,}))$  constructed from the bilinear form in \ref{sec-parabolic-subgroups}(ii) and the Cartan involution. So in fact $\ad_Y{=}0$.
\end{proof}

\begin{proof}[Proof of Theorem~\ref{thm-normalizer-of-h-i}]  
It follows from Lemma~\ref{lem-nilpotent-computation} that 
\[
\overline{\lie{n}}_I = [\lie{g}, \lie{g}]
\,\cap\, \{\, X\in \lie{h}_I : \text{$\ad_Z\colon \lie{g}\to \lie{g}$ is nilpotent}\,\}.
\]
and therefore 
\[
\operatorname{Normalizer}_G(\lie{h}_I)
\subseteq 
\operatorname{Normalizer}_G(\overline{\lie{n}}_I).
\]
But according to Theorem~\ref{thm-normalizer-thm-of-moore}, 
\[\operatorname{Normalizer}_G(\overline{\lie{n}}_I)
= \overline{P}_I = \theta[ P_I].
\]
So  to prove the first assertion in the statement of the theorem it suffices to prove that  $\operatorname{Normalizer}_{\overline{P}_I}(\lie{h}_I) = A_IH_I$.

Let $g=ma\overline n$ be an element of  $\overline P_I=M_IA_I\overline N_I$ that normalizes $\lie{h}_I$.  Since $a$ and $\overline n$ normalize $  { \lie {h}_I}$, the element $m$ must normalize $\lie{h}_I$, too.  Since $\lie{k}_I =  \lie{m}_I\cap \lie {h}_I$,   the element $m$ must normalize not only $ \lie{h}_I$ but $\lie{k}_I$. But $K_I$ is a maximal compact subgroup in $M_I$, and the normalizer of $K_I$ in $M_I$ is $K_I$ itself.    So $m\in K_I$ and therefore $g=ma\overline{n}\in A_IH_I$, as required.

The second assertion is proved in the same way: an element of $G$ that conjugates $\lie{h}_I$ into $\lie{h}_J$ must conjugate $\overline{n}_I$ into $\overline{n}_J$, and from here we may apply Theorem~\ref{thm-tits-theorem}.
\end{proof}

\begin{remark}
    It follows from the second assertion in the theorem   that $H_I$ and $H_J$ are conjugate if and only if $I {=} J$. Moreover the theorem  also implies that if $t,t'\in\R^\Sigma$, then $\lie{h}_t$ and $\lie{h}_{t'}$ are conjugate  if and only if for every $\alpha \in \Sigma$,
\begin{equation}\label{eq: conjugate of h_t}
    t_\alpha \ne 0\quad \Leftrightarrow \quad t'_\alpha \ne 0,
\end{equation}
 and that $H_t$ and $H_{t'}$ are conjugate if and only if \eqref{eq: conjugate of h_t} holds.
\end{remark}

\begin{definition}
Let  $\pmb{\mathfrak{g}} = \mathfrak{g} {\times } \R^\Sigma$ be  the  constant vector bundle over $\R^\Sigma$ with fiber $\mathfrak{g}$, and denote by $\pmb{\mathfrak{h}}$ the smooth vector subbundle of  $\pmb{\mathfrak{g}}$ whose fiber over $t\in \R^\Sigma$ is    the Lie subalgebra  $\mathfrak{h}_t$. Thus 
$\pmb{\mathfrak{h}} = \{ \, (X,t) \in \pmb{\mathfrak{g}} : X \in \mathfrak{h}_t\,\}$.
\end{definition}

\begin{proof}[Proof of Theorem~\ref{thm:H}]
As in the statement of the theorem, denote  by $\bigH\subseteq \bigG$ the  closure of the subset 
\[
\{\, (h,t) :h \in H_t,\,\,\,  t \in \R^\Sigma_\times \,\} \subseteq \bigG  
\]
We shall temporarily denote by $H'_t \subseteq G$ the fiber of $\bigH$ over $t\in \R^\Sigma$.  It is a closed subgroup of $G$. In fact, using the split decomposition $G=M_\Sigma A_\Sigma$, every $H'_t$ is a subgroup of $M_\Sigma$.

Our first objective is to show that $H'_t$  is equal to the subgroup $H_t$ from Definition~\ref{def-of-all-h-t-groups}. 
Fix $t_0\in \R^\Sigma$. Every element of $H_{t_0}$ may be written as a product 
\begin{equation}
    \label{eq-expression-for-h-in-h}
h = m \exp(X_1)\cdots \exp (X_k),
\end{equation}
for some $m\in M$ and some elements $X_1,\dots, X_k\in \lie{h}_{t_0}$. By extending the $X_j$ to smooth sections of $\pmb{\lie{h}}$, we find that there  is a smooth section 
\begin{equation}
    \label{eq-smooth-section-of-h-bundle}
\begin{gathered}
\R^\sigma \longrightarrow  \bigG
\\
t \longmapsto  m \exp(X_{1,t})\cdots \exp (X_{k,t})
\end{gathered}
\end{equation}
whose value at every $t$ lies in $H_t$, and whose value at $t_0$ is $h$.  This shows that $H_{t_0}\subseteq H'_{t_0}$.

To prove that $H'_t\subseteq H_t$, for every $t\in \R^\Sigma$, observe that by continuity of the adjoint action, $H'_t$ certainly lies within the normalizer of $\lie{h}_t$, for every $t$. And therefore it lies within $A_I H_t$, for $I = \{ \, \alpha \in \Sigma: t_\alpha \ne 0\,\}$, and so in $M_\Sigma \cap A_I H_t$.

Now, the spectrum of every element of $K$, viewed as a linear operator on the finite-dimensional complex vector space $\lie{g}_{\C}$ via the adjoint action, lies within the unit circle in $\C$.  Since the spectrum is conjugation invariant, the same is true of any element of $H_t$, for any $t\in \R^\Sigma_\times$, and since in finite dimensions the spectrum varies continuously with the operator, the same is also true of any element in any fiber $H'_t$ of $\bigH$. But the only elements of $M_\Sigma \cap A_I H_t$ whose  spectrum is included in the unit circle are those in the subgroup $H_t$.   Therefore $H'_t \subseteq H_t$, as required. 

It remains to prove that $\bigH$ is a smooth submanifold of $\bigG$.  The exponential map  
$\exp \colon \pmb{\lie{g}}\!\to\!   \bigG$ 
is a diffeomorphism from a neighborhood of the zero section in $\pmb{\lie{g}}$ to a neighborhood of the group-identity  section in $\bigG$. It restricts to a homeomorphism from an open neighborhood of the zero section in $\pmb{\lie{h}}$ to an open neighborhood of the group-identity  section in $\bigH$. So the latter  neighborhood is a smooth submanifold of $\bigG$. Pointwise multiplication by the sections \eqref{eq-smooth-section-of-h-bundle} then shows that any point of $\bigH$ has an open neighborhood in $\bigH$ that is a smooth submanifold of $\bigG$.
\end{proof}

 \begin{example} 
\label{ex-sl2-part-zero}
Let $G=SL(2,\R)$.
Equip $G$ with  the  Iwasawa decomposition  for  which 
     \[
     K = SO(2),\quad 
     A = \left \{\, 
     \left [ \begin{smallmatrix} e^x & 0 \\ 0 & e^{-x}\end{smallmatrix}
     \right ]
     : x\in \R
     \,\right \} \quad \text{and} \quad 
     N = \left \{\, 
     \left [ \begin{smallmatrix} 1 & y \\ 0 &1\end{smallmatrix}
     \right ]
     : y \in \R
     \,\right \}.
     \]
There  is just one simple restricted root, which is $\alpha \colon \left [ \begin{smallmatrix} x & 0 \\ 0 & -x\end{smallmatrix}\right ] \mapsto 2x$. If we    identify $\R^\Sigma$ with $\R$ accordingly, then 
\[
a_t = \left [ \begin{smallmatrix} t^{1/2} & 0 \\ 0 & t^{-1/2} \end{smallmatrix}
     \right ]
\quad \text{and} \quad 
H_t = \left \{\, 
      \left [ \begin{smallmatrix} c & -t^2s \\ s\strut & c \end{smallmatrix}
     \right ]
     : c^2 + t^2s^2 =1  
     \,\right \}.
\]
\end{example}

\subsection{Construction of the Oshima space}
\label{sec-action-of-a}

We first assemble the family of homogeneous spaces $\{\, G/H_t: t\in \R^\Sigma\,\} $  into a single smooth manifold $\bigG/\bigH$ using following well-known result:

\begin{lemma}[Godement's criterion; see for instance {\cite[Sec.\,4.9]{GodementLieGroups}}]
\label{lem-godement}
Let $W$ be a smooth manifold and let $R\subseteq W{\times} W$ be an equivalence relation.  There is a smooth manifold structure on $W/R$ for which the quotient map $W\to W/R$ is a smooth submersion if and only if 
\begin{enumerate}[\rm (i)]

\item $R$ is a closed subset of  $W{\times} W$ and a smooth embedded  submanifold, and 

\item the projection map $R\to W$ onto either factor of $W$ in $W{\times}W$ is a smooth submersion. 

\end{enumerate}
\end{lemma}

\begin{definition} 
We shall  denote by   $\bigG/\bigH$ the quotient of $\bigG= G{\times}\R^\Sigma$ by the equivalence relation
\[
(g_1,t_2) \sim (g_2,t_2) 
\quad \Leftrightarrow \quad 
t_1=t_2\quad \&\quad 
g_2^{-1}g_1 \in H_{t_1},
\]
and we shall denote by $[g,t]\in \bigG/\bigH$ the equivalence class of $(g,t)\in \bigH$. 
\end{definition}

By Godement's criterion, the quotient $\bigG/\bigH$ carries a unique smooth manifold structure for which the quotient map from $\bigG$  is a submersion.  Oshima's space is obtained from $\bigG/\bigH$ by a further quotient an action of the group $A$.

Thanks to Lemma \ref{lem-h-conjugation},
the subgroups $H_t$ have the following conjugation equivariance property: if $a\in A$ and $t\in\R^\Sigma$, then
\begin{equation}
\label{eq:H-conjugation}
    H_{a\cdot t} = \Ad_a [H_t].
\end{equation}
In other words, for any $a\in A$  we have
\begin{equation}
    \label{eq:H-A-action}
    (h,t) \in \bigH \quad\Rightarrow\quad (a^{-1}ha,a^{-1}\cdot t) \in \bigH.
\end{equation}
From this it follows that the right action of $A$ on $\bigG$ given by
\begin{equation}
    \label{eq:Oshima-A-action}
    (g,t)\cdot a = (ga,a^{-1}t), \qquad (a\in A, \ [g,t] \in \bigG/\bigH)
\end{equation}
descends to a smooth right action of $A$ on $\bigG/\bigH$. 

\begin{definition}
The \emph{Oshima space} of the real reductive group $G$ (with the chosen Iwasawa decomposition $KAN$) is the quotient space 
\begin{equation}
    \Oshima = (\bigG/\bigH)/A.
\end{equation}
We shall denote by $\ldoublebracket g,t\rdoublebracket\in \Oshima$ the class in $\Oshima$ determined by a point $[g,t]\in \bigG/\bigH$.
\end{definition}
The following two lemmas generalize the Iwasawa decomposition $NAK=G$ to a fiberwise product decomposition $NA\bigH \hookrightarrow \bigG$ of an open subset of $\bigG$. 

\begin{lemma} 
\label{lem-trivial-intersection}
 For every $t\in \R^\Sigma$,   $H_t\cap AN$ is the trivial subgroup of $G$.
\end{lemma}

\begin{proof} 
It suffices to prove the lemma when $H_t = H_I$, for some $I\subseteq \Sigma$. 
In this case there exists $X\in \mathfrak{a}_I$ such that, if $ \overline n \in \overline{\mathfrak{n}}_I$, then 
\[
\lim_{j\to +\infty} \exp(jX) \overline{n} \exp (-jX)  = e,
\]
while of course $\exp(jX)\in A_I$ commutes with $K_I$.  Since $\exp(X)$ conjugates $AN$ into itself, it follows that 
\[
k_I\overline{n}_I \in H_I\cap AN \quad \Rightarrow \quad k_I\in AN.
\]
Since $AN$ has no nontrivial compact subgroup, it follows that 
\[
H_I\cap AN = \overline{N}_I\cap AN .
\]
But $\overline{N}_I\subseteq \overline{N}$, and $\overline{N}\cap AN$ is well known to be trivial (this fact is also readily proved by the  method above).
\end{proof}

\begin{lemma}
    \label{lem:NAH}
    The multiplication map
    \begin{gather*}
      N   \times A \times \bigH   \longrightarrow \bigG   
      \\
      (n,a,(h,t)) \longmapsto (nah, t) 
    \end{gather*}
     is a diffeomorphism onto an open subset of $\bigG$.
\end{lemma}

\begin{proof} 
The map is injective thanks to Lemma~\ref{lem-trivial-intersection} and of full rank near $n=a=h=e$ thanks to  the direct sum decomposition formula 
\[
\mathfrak{h}_t\oplus \mathfrak{a} \oplus \mathfrak{n} = \mathfrak{g},
\]
which is clear from the definition of $\mathfrak{h}_t$. Using left and right translations in $G$, we find that it  is of full rank everywhere; compare \cite[Lemma\,6.44]{KnappBeyond}.  
\end{proof} 

\begin{proposition}
    \label{prop:AH_closed}
    The map 
    \begin{gather*}
       A  \times \bigH \longrightarrow   \R^\Sigma \times \bigG 
     \\ 
      (a,h,t) \longmapsto  (t, ha,a^{-1}\cdot t )   
    \end{gather*}
    is a closed embedding  of smooth manifolds.
\end{proposition}

\begin{proof}
Lemma~\ref{lem:NAH} shows that the map is an injective immersion and  a homeomorphism onto its image.  So it remains  to show the image is    closed.  

The elements $(t', g, t'')$ in the image of the map satisfy 
\begin{equation}
    \label{eq-closure-conjugation-relation}
\Ad_{g^{-1}} [ \lie{h}_{t'} ]= \lie{h}_{t''} ,
\end{equation}
and by continuity the same relation holds for points in the closure. This being so, it follows from Proposition~\ref{prop:h_t_structure} that if $(t',g,t'')$ is a point in the closure then
for all $\alpha\in \Sigma$, 
\[
t'_\alpha \ne 0 \quad \Leftrightarrow \quad t''_\alpha \ne 0.
\]
Let $I =\{\, \alpha \in \Sigma: t'_\alpha \ne 0\,\}$.  It follows from continuity that   $\operatorname{sign} (t'_\alpha) = \operatorname{sign}(t''_\alpha)$ for all $\alpha \in I$, and therefore $t''=a^{-1} t'$ for some $a\in A$. Let us now write 
\[
(t', g, t'') = (t', \gamma a,a^{-1}\cdot t'),
\]
where $\gamma = ga^{-1}$.  It follows from \eqref{eq-closure-conjugation-relation} that 
\[
\Ad_{a^{-1}\gamma^{-1}}[\lie{h}_{t'}] =  \lie{h}_{t''}
= \lie{h}_{a^{-1}t'}  = \Ad_{a^{-1}}[\lie{h}_{t'}]
\]and therefore $\gamma^{-1} $ normalizes $\lie{h}_{t'}$.  By Theorem \ref{thm-normalizer-of-h-i}, $\gamma\in H_{t'}A_I$, so $\gamma=h\alpha$ for some $h \in \lie{h}_{t'}$, $\alpha \in A_I$.
As a result
\[
(t', g, t'')) = (t', h  \alpha  a,a^{-1}\cdot t') = (t', h  \alpha  a,(a\alpha)^{-1} \cdot t'),
\]
since $A_I$ fixes $t'$.  We have therefore proved that the image of our map is closed. 
\end{proof}

We can now prove that the main theorem of this section.

\begin{theorem}
\label{thm-proper-a-action}
    There is a unique smooth manifold structure on the Oshima space $\Oshima$ for which the quotient map from $\bigG$ to $\Oshima$ is a smooth submersion.
\end{theorem}

\begin{proof}
The space $\Oshima$ is the quotient of $\bigG$ by the equivalence relation 
\[
\bigl \{ \,
\bigl ( (g,t), (gha,a^{-1}\cdot t)\bigr )\in \bigG\times \bigG : g\in G,\,\, t\in \R^\Sigma,\,\, h\in H_t,\,\, a\in A
\, \bigr \}.
\]
It follows from Proposition~\ref{prop:AH_closed} that this is a closed subset and an embedded submanifold, and it therefore follows from Godement's criterion that there is a unique smooth manifold structure on $\Oshima$ for with the quotient map $\bigG\to \Oshima$ is a smooth submersion.
\end{proof}

\begin{remarks}
    (i) The  theorem is  equivalent to the assertion that there is a unique smooth manifold structure on $\Oshima$ for which the quotient map from $\bigG/\bigH$ to $\Oshima$ is a smooth submersion (but it is a bit more convenient to prove the theorem as stated). 
\label{rmk:groupoid_construction}
    (ii) In group action language, the action of $A$ on $\bigG/\bigH$ is free and proper.   To take this remark even further, the smooth family of Lie groups $\bigH$ is a Lie groupoid over $\R^\Sigma$, and  \eqref{eq:H-A-action} gives an action of $A$ on $\bigH$ by Lie groupoid automorphisms.  The Oshima space is then the quotient of the space $\bigG\to\R^\Sigma$ by a free and proper right action of the Lie groupoid $\bigH\rtimes A$.  
\end{remarks}

\begin{theorem}[Oshima {\cite{Oshima78}}]
\label{thm:Oshima-compactness}
The  manifold $\Oshima$ is compact.
\end{theorem}

\begin{proof} 
The set 
\begin{equation}
    \label{eq-compact-set-in-oshima}
\{\, \ldoublebracket k, t\rdoublebracket \in \Oshima : k\in K ,\,\,\, | t_\alpha | {\le} 1 \,\, \forall \alpha \in \Sigma \,\}
\end{equation}
is a compact subset of $\Oshima$ (it is the continuous image of a compact set), while the set 
\begin{equation}
    \label{eq-dense-set-in-oshima}
\{ \, \ldoublebracket g,t\rdoublebracket\in \Oshima : g\in G,\,\,  t_\alpha{\ne} 0\,\,\, \forall \alpha \in \Sigma\,\} 
\end{equation}
is dense in $\Oshima$. So it suffices to prove that \eqref{eq-dense-set-in-oshima} is a subset of \eqref{eq-compact-set-in-oshima}.

Using the $A$-action,   \eqref{eq-dense-set-in-oshima} is equal to 
\[
\{ \, \ldoublebracket g,t\rdoublebracket\in \Oshima : g\in G, \,\, t_\alpha {=}\pm 1,\, \forall \alpha \in \Sigma\,\} .
\]
Using the $KAK$ decomposition of $G$ \cite[Sec.\,VII.3]{KnappBeyond}, and keeping in mind that $H_{t} {=}K$ if $t_\alpha {=} \pm 1$ for all $\alpha$, the above  is in turn  equal to
\begin{equation}
    \label{eq-dense-set-in-oshima-2}
\{ \, \ldoublebracket ka,t\rdoublebracket\in \Oshima : k\in K,\,\, a\in {A^{-}}, \,\,\, t_\alpha {=}\pm 1\,\, \forall \alpha \in \Sigma\,\} ,
\end{equation}
where $A^-$ is the closed negative chamber in $A$:
\[
A^- = \{\, a\in A : a^\alpha {\le}  1\,\, \forall \alpha \in \Sigma \,\}.
\]
Using the $A$-action again, we find that \eqref{eq-dense-set-in-oshima-2} is a subset  \eqref{eq-compact-set-in-oshima}, as required.
\end{proof}

\subsection{Orbits in the Oshima space}
\label{sec:orbits}

The Oshima space is equipped with a smooth left action of $G$ by left-multiplication: 
\[
\gamma\cdot  \ldoublebracket g,t\rdoublebracket = \ldoublebracket \gamma g,t\rdoublebracket\qquad \forall \gamma, g\in G,\,\, \forall t\in \R^\Sigma.
\]
There are only finitely many $G$-orbits in $\Oshima$. To see this, observe that the continuous map  
\begin{equation}
    \label{eq-counting-orbits-in-oshima}
    \begin{gathered}
    \Oshima \longrightarrow \R^\Sigma / A
    \\
    \ldoublebracket g,t\rdoublebracket \longmapsto [t]
    \end{gathered}
\end{equation}
passes to a bijection from the  set of $G$-orbits in $\Oshima$ to the set of $A$-orbits in $\R^\Sigma$.  Thus, the orbit of a point $\ldoublebracket g,t \rdoublebracket$ is determined by the signs of the components $t_\alpha$, $\sign(t_\alpha)\in\{-1,0,1\}$.

\begin{lemma}
    \label{lem-orbit-structure}
    Let $\ldoublebracket g,t \rdoublebracket\in\Oshima$ and let $I=\{\alpha\in\Sigma : t_\alpha\neq 0\}$ as in \eqref{eq:I}.  Then the orbit of $\ldoublebracket g,t \rdoublebracket$ is $G$-equivariantly diffeomorphic to the homogeneous space $G /   A_IH_I$.      
\end{lemma}

\begin{proof} 
There is a point of the form $\ldoublebracket e,t\rdoublebracket$ in the orbit for which 
\[
t_{\alpha} = \begin{cases}  \pm 1 & \alpha \in I \\ 0 & \alpha \notin I.
\end{cases}
\]
The  subgroup $H_{t}$  is equal to $H_I$.  Now 
\[
g \cdot \ldoublebracket e,t\rdoublebracket = \ldoublebracket e,t\rdoublebracket
\quad 
\Leftrightarrow \quad 
\exists h \in H_I,\,\,\, \exists a\in A: \,\,\, 
g = ha \,\,\, \& \,\,\, a^{-1} \cdot t= t.
\]
 But the condition $a^{-1}\cdot t = t$ is equivalent to $a \in A_I$, so 
\[
g \cdot \ldoublebracket e,t\rdoublebracket = \ldoublebracket e,t\rdoublebracket
\quad 
\Leftrightarrow \quad 
g \in A_I N_I,
\]
as required.
\end{proof}

\begin{example} 
\label{ex-sl2-part-two}
Let $G=SL(2,\R)$,
equipped with the Iwasawa decomposition described in Example~\ref{ex-sl2-part-zero}.  If $\C_\infty$ denotes the extended complex plane (\emph{i.e}, the Riemann sphere), then the map 
\[
\begin{gathered}
\bigG \longrightarrow \C_\infty\\
(g,t) \longmapsto g\cdot (it) \in \C_\infty,
\end{gathered}
\]
where the dot indicates action by Mobius transformations, factors as the composition of the submersion $\bigG\to \Oshima$ with a diffeomorphism  
\[
\begin{gathered}
\Oshima \stackrel \cong \longrightarrow \C_\infty\\
\ldoublebracket g,t\rdoublebracket \longmapsto g\cdot it \in \C_\infty.
\end{gathered}
\]
One sees that there are three $G$-orbits, namely the upper and lower half-planes, both diffeomorphic to $G/H_\Sigma = G/K$, and the extended real line $\R\sqcup\{\infty\}$, diffeomorphic to $G/H_\emptyset A_\emptyset = G/MA\overline{N}$. 
\end{example}

\subsection{Group bundle over the Oshima space}

It follows from the definition of equivalence relation defining the Oshima space  that if $g_1,g_2\in G$ and  $t_1,t_2\in \R^\Sigma$, then 
\[
\ldoublebracket g_1,t_1\rdoublebracket = \ldoublebracket g_2,t_2\rdoublebracket \quad \Rightarrow \quad  \Ad_{g_1}[ H_{t_1}] =  \Ad_{g_2 }[ H_{t_2}] .
\]
\begin{definition}
    \label{def-h-g-t}
For $g\in G$ and $t\in \R^\Sigma$ we shall write 
\[
H_{\ldoublebracket g,t\rdoublebracket} = \Ad_g [ H_t],
\]
which is a closed subgroup of $G$.  We then write 
\begin{equation*}
    \HBundle_{\Oshima} = \bigl \{ \,
    \bigl (\ldoublebracket g,t\rdoublebracket, \gamma \bigr )\in \Oshima \times G : \gamma \in H_{\ldoublebracket g,t\rdoublebracket} \, \bigr \} .
\end{equation*}
\end{definition}

\begin{theorem} 
\label{thm-smooth-bundle-of-groups-over-oshima}
The space $\HBundle_{\Oshima}$ is a closed  submanifold of $\Oshima\times G$, and the projection map to $\Oshima$ is a smooth submersion.
\end{theorem}

A simple way prove the theorem is to introduce the following coordinate patches on $\Oshima$: 

\begin{lemma} 
\label{lem-big-cell}
The smooth map 
\[
\begin{gathered} 
N \times \R^\Sigma \longrightarrow \Oshima 
\\
(n,t) \longmapsto \ldoublebracket n,t\rdoublebracket
\end{gathered}
\]
is a diffeomorphism onto an open subset $\BigCell$ of $\Oshima$.
\end{lemma}

\begin{proof} 
It follows from Lemma~\ref{lem-trivial-intersection} that the map is injective.
By composing the diffeomorphism in  Lemma~\ref{lem:NAH}
with the self-diffeo\-morphism
\begin{gather*}
       N \times A \times  \bigH   \longrightarrow  N \times A \times  \bigH   
      \\
      (n,a,(h,t)) \longmapsto (n,a,  (a^{-1}ha, a^{-1}t))  ,
\end{gather*}
we find that 
the map 
\begin{equation}
        \label{eq-big-cell-map}
\begin{gathered}
      N   \times A \times \bigH   \longrightarrow \bigG   
      \\
      (n,a,(h,t)) \longmapsto (nha, a^{-1}t) 
    \end{gathered}
\end{equation} 
is a diffeomorphism onto an open subset.  The composition of this, in turn, with the submersion $\bigG \to \Oshima$ is therefore   a submersion onto its image. Considering the commuting diagram
\[
\xymatrix@C=50pt{
 N   \times A \times \bigH \ar[d]  \ar[r]^-{\eqref{eq-big-cell-map}} &  \bigG   \ar[d]
 \\
 N\times \R^{\Sigma} \ar[r] & \Oshima ,
}
\]
in which the left-vertical morphism is the obvious projection map and the bottom map is the map in the statement of this lemma, we see that the bottom map is also a submersion onto its image, and hence, since 
\[
\dim (N) + \dim (\R^\Sigma) = \dim (\Oshima),
\]
it is a diffeomorphism onto its image.
\end{proof}
 
Obviously we can now cover $\Oshima$ with open sets of the form 
\[
\{\, \ldoublebracket gn,t\rdoublebracket : n\in N, t\in \R^\Sigma\,\}
\]
for varying $g\in G$.

\begin{proof}[Proof of Theorem \ref{thm-smooth-bundle-of-groups-over-oshima}]
After pulling back $\HBundle_\Oshima$ along the diffeomorphism in Lemma~\ref{lem-big-cell}, we obtain the space 
\[
 \bigl \{ \,
      (n,t, \gamma  )\in N\times \R^\Sigma \times G : \gamma \in \Ad_n [ H_t] \, \bigr \}  \subseteq N\times \R^\Sigma \times G.
\]
Under the self-diffeomorphism 
\[
\begin{gathered} 
N\times \R^\Sigma \times G  \longrightarrow N\times \R^\Sigma \times G 
\\
(n,t,\gamma) \longmapsto (n,t,n^{-1}\gamma n)
\end{gathered}
\]
this is in turn mapped to the subspace 
\begin{equation}
    \label{eq-conjugated-bundle-over-big-cell}
 \bigl \{ \,
      (n,t, \gamma  )\in N\times \R^\Sigma \times G : \gamma \in   H_t \, \bigr \} \subseteq N\times \R^\Sigma \times G.
\end{equation}
But \eqref{eq-conjugated-bundle-over-big-cell} is evidently a closed subset and a smooth submanifold.  The theorem follows from this.   
\end{proof}

\subsection{Embedding the Satake compactification into the Oshima space}
\label{sec-satake-in-oshima}
Let us now return to the Satake compactification $\Satake$, as defined in Section~\ref{sec:Mohsen}. 

\begin{definition}
    \label{def-oshima-plus}
Let $G=KAN$ be a  real reductive group, and let $\Oshima$ be the associated Oshima space. We shall write    
    \[
   \Oshima_{+} =  \{ \, \ldoublebracket g,t\rdoublebracket \in \Oshima : t_\alpha > 0 \,\,\,\forall \alpha \in \Sigma \,\}   ,
    \]
and then denote by  $\overline{\Oshima_{+}}$  the closure of $\Oshima_+$  in $\Oshima$.
\end{definition}

The space $\Oshima_+$ is an open subset and a single $G$-orbit, of the type   $G/K{\cdot} A_\Sigma$.

\begin{theorem}
\label{thm-satake-in-oshima}
 The  map 
\begin{equation}
    \label{eq:Satake_map}
\begin{gathered}
    \Oshima \longrightarrow \Subgp(G) 
    \\
   \ldoublebracket g,t\rdoublebracket\mapsto H_{\ldoublebracket g,t\rdoublebracket} 
\end{gathered}
\end{equation}
restricts to a $G$-equivariant homeomorphism from $\overline{\Oshima_{+}}$ to $\Satake$.
\end{theorem}

\begin{proof}
Let us first show that the restricted map $\overline{\Oshima_{+}}\to \Satake$ is injective. According to Proposition~\ref{prop:h_t_structure}, if 
$H_{\ldoublebracket g',t'\rdoublebracket} = H_{\ldoublebracket g'',t''\rdoublebracket}$ then for every $\alpha\in\Sigma$, either $t'_\alpha$ and $t''_\alpha$ are both zero or both strictly positive.  Thus 
\[
H_{\ldoublebracket g',t'\rdoublebracket} = H_{\ldoublebracket g'',t''\rdoublebracket} \quad \Rightarrow \quad 
\exists a\in A:  a \cdot t'=t'' .
\]
Therefore, setting $g{=}g''a$, we need only show that 
\[
H_{\ldoublebracket g',t'\rdoublebracket} = H_{\ldoublebracket g,t'\rdoublebracket} \quad \Rightarrow \quad 
\ldoublebracket g',t'\rdoublebracket = \ldoublebracket g,t'\rdoublebracket.
\]
But 
\[
H_{\ldoublebracket g',t'\rdoublebracket} = H_{\ldoublebracket g,t'\rdoublebracket} \quad \Rightarrow \quad 
\Ad_{g^{-1}g'}[H_{t'}] = H_{t'},
\]
and the required result therefore follows from Theorem~\ref{thm-normalizer-of-h-i}. 

Next, Theorem \ref{thm-smooth-bundle-of-groups-over-oshima} and Lemma \ref{lem:bundle} imply that the map  $\Oshima \to \Subgp(G)$ is continuous.  Since $\overline{\Oshima_+}$ is compact, and the image of $\Oshima_+$ in $\Satake$ is the dense open subset of all maximal compact subgroups, the image of $\overline{\Oshima_+}$ 
is $\Satake$. Being a bijection from a compact space, the map is a homeomorphism, and $G$-equivariance is evident from   Definition~\ref{def-h-g-t}.
\end{proof}

Theorem~\ref{thm-satake-in-oshima} shows that  that the Satake compactification decomposes into finitely many $G$-orbits, as in Section \ref{sec:orbits}.   To be specific, for $I\subseteq \Sigma$ let $t_I\in \R^\Sigma$ be defined by 
\begin{equation}
    \label{eq-def-of-t-i}
t_{I,\alpha} = \begin{cases}   1 & \alpha \in I \\ 0 & \alpha \notin I,
\end{cases}
\end{equation}
and let $x_I$ be the image under \eqref{eq:Satake_map} of $\ldoublebracket e,t_I \rdoublebracket \in \overline{\Oshima_+}$.
Using \eqref{eq-counting-orbits-in-oshima} shows that 
$\{\,x_I : I \subseteq \Sigma \,\}$ is a complete set of orbit representatives for 
the Satake compactification. By Lemma~\ref{lem-orbit-structure}, the $G$-orbit $\Satake_I$ of $t_I$ is isomorphic to $G /   A_IH_I$, and there is therefore 
 an isomorphism of $G$-sets
\begin{equation}
    \label{eq-orbit-decomposition-of-satake}
 \Satake =  \bigsqcup_{I\subseteq \Sigma} \Satake_I \cong \bigsqcup_{I\subseteq \Sigma}  G /   A_IH_I 
 \end{equation} 
 Concerning the topology on this disjoint union, it follows from the continuity of the map \eqref{eq-counting-orbits-in-oshima}, and a simple additional computation, that the orbit $\Satake_I\subseteq \overline{\Satake_J}$ if and only if $I\subseteq J$.

\section{The Oshima groupoid} 
\label{sec-oshima-groupoid}

In this section we shall construct our \emph{Oshima groupoid}, $\Groupoid_{\Oshima}$,  whose space of units is the Oshima space.   The reduction of $\Groupoid_\Oshima$ to the Satake compactification $\overline{\Oshima_+} \cong \Satake$ will be homeomorphic to the Satake groupoid $\Groupoid_\Satake$ of Definition \ref{def:Satake_groupoid}, thus endowing the Satake groupoid with a smooth structure.

\subsection{Normal subgroupoids and   quotients of Lie groupoids}

\begin{definition}[See for instance {\cite[Sec.\,2.2]{Mackenzie}}]
\label{def-normal-lie-sugroupoid}
Let $\Groupoid$ be a Lie groupoid. A  \emph{normal Lie subgroupoid}
of $\Groupoid$ is a Lie subgroupoid $\Subgroupoid $ with the following properties:
    \begin{enumerate}[\rm (i)]

    \item the inclusion $\Subgroupoid\to \Groupoid$ is a closed embedding of smooth manifolds,

    \item $\Subgroupoid$ includes every identity morphism of $\Groupoid$,
    
    \item $\Subgroupoid$ consists  entirely of isotropy elements (that is, elements whose source and target are equal), and 
    
    \item if  $\alpha\in\Subgroupoid$ and  $\gamma\in\Groupoid$, and if   $\operatorname{target}(\alpha)=\operatorname{source}(\gamma)$, then  $\gamma\circ \alpha\circ \gamma^{-1} \in \Subgroupoid$.
    \end{enumerate}
\end{definition}

If $\Subgroupoid$ is a normal Lie subgroupoid of a Lie groupoid $\Groupoid$, then we may place on $\Groupoid$ the equivalence relation 
\[
 \gamma \simeq \eta 
 \quad \Leftrightarrow \quad
 \operatorname{target}(\gamma) = \operatorname{target}(\eta) \,\,\, \text{and} \,\,\, \gamma^{-1} \eta \in \cH.
\]
The set $\Groupoid/\Subgroupoid$ of equivalence classes is a groupoid with the same object space as $\Groupoid$ and the structure maps 
\[
\operatorname{source}([\gamma]) = \operatorname{source}(\gamma),\quad 
\operatorname{target}([\gamma]) = \operatorname{target}(\gamma)\quad \text{and} \quad 
[\gamma_1]\circ [\gamma_2] = [\gamma_1\circ \gamma_2].
\]

\begin{lemma}
   If $\Subgroupoid$ is a normal Lie subgroupoid of a Lie groupoid $\Groupoid$, then there is a unique smooth manifold structure on $\Groupoid / \Subgroupoid$ for which the quotient map $\Groupoid\to \Groupoid/\Subgroupoid$ is a submersion.  With this smooth manifold structure, $\Groupoid/\Subgroupoid$ is a Lie groupoid.
\end{lemma}

\begin{proof} 
This is a consequence of Godement's criterion, Lemma~\ref{lem-godement}.
\end{proof}

\subsection{The transformation groupoid and the Oshima groupoid}
 The transformation groupoid associated to the action of $G$ on the Oshima space $\Oshima$ is of course 
\[
G \ltimes \Oshima = \bigl \{\, 
\bigl ( \ldoublebracket g_2,t_2\rdoublebracket, g , \ldoublebracket g_1,t_1\rdoublebracket \bigr ) \in \Oshima\times G \times \Oshima : \ldoublebracket g_2, t_2 \rdoublebracket = g {\cdot}\ldoublebracket g_1 ,t_1\rdoublebracket
\,\bigr \},
\]
with the structure maps 
   \[
\begin{aligned}
\operatorname{source}\bigl ( \ldoublebracket g_2,t_2\rdoublebracket, g , \ldoublebracket g_1,t_1\rdoublebracket \bigr )  
    & = \ldoublebracket g_1,t_1\rdoublebracket
\\
\operatorname{target} \bigl ( \ldoublebracket g_2,t_2\rdoublebracket, g , \ldoublebracket g_1,t_1\rdoublebracket \bigr )
    & = \ldoublebracket g_2,t_1 \rdoublebracket
\\
 \bigl ( \ldoublebracket g_3,t_3\rdoublebracket, g'' , \ldoublebracket g_2,t_2\rdoublebracket \bigr ) \circ  \bigl ( \ldoublebracket g_2,t_2\rdoublebracket, g' , \ldoublebracket g_1,t_1\rdoublebracket \bigr ) & = \bigl( \ldoublebracket g_3,t_3 \rdoublebracket, g''g', \ldoublebracket g_1,t_1 \rdoublebracket\bigr )   .
\end{aligned}
\]
The smooth famliy of groups $\Subgroupoid_{\Oshima}$ in Definition~\ref{def-h-g-t} may be viewed as a normal Lie subgroupoid of the transformation groupoid: 

\begin{lemma}
\label{lem-normal-subgroupoid-of-transformation-groupoid}
The subspace  
\begin{equation*}
\Subgroupoid_{\Oshima} \cong \bigl \{\, \bigl (\ldoublebracket  g,t\rdoublebracket, h , \ldoublebracket g,t\rdoublebracket \bigr ) :  \ldoublebracket g,t\rdoublebracket\in \Oshima,\,\, h\in H_{\ldoublebracket g,t\rdoublebracket}\,\bigr \} 
\subseteq G \ltimes \Oshima
\end{equation*}
is a normal Lie subgroupoid of the transformation groupoid. 
\end{lemma}

\begin{proof}
    Condition (i) in Definition~\ref{def-normal-lie-sugroupoid} is a consequence of Theorem~\ref{thm-smooth-bundle-of-groups-over-oshima}.   Conditions (ii) and (iii) are clear, and condition (iv) follows from the conjugation relation  $\Ad_{\gamma}H_{\ldoublebracket g,t\rdoublebracket} =H_{\gamma {\cdot} \ldoublebracket  g,t\rdoublebracket} $. 
\end{proof}

\begin{definition}
    \label{def-oshima-groupoid}
    Let $G=KAN$ be a real reductive group, and let $\Oshima$ be its Oshima space. The \emph{Oshima groupoid} $\Groupoid_{\Oshima}$ is the quotient of the transformation groupoid $G \ltimes \Oshima$ by the normal Lie subgroupoid $\Subgroupoid_{\Oshima}$ in Lemma~\ref{lem-normal-subgroupoid-of-transformation-groupoid}. 
\end{definition}

\subsection{Embedding the Satake groupoid into the Oshima groupoid}
\label{sec:SatakeGroupoidinOshimaGroupoid}

\begin{definition}
\label{def-notation-for-elements-of-oshima-groupoid}
    We shall use the notation 
    \[
    \ldoublebracket g,g_1,t\rdoublebracket \in \Groupoid_\Oshima
    \]
    to indicate the equivalence class  of the morphism 
$
\bigl ( \ldoublebracket g_2,t\rdoublebracket, g , \ldoublebracket g_1,t\rdoublebracket \bigr ) 
$
in $G\ltimes \Oshima$.
\end{definition}

\begin{theorem}
    \label{thm-satake-groupoid-is-oshima-groupoid}
    The map 
    \[
    \Groupoid_{\Oshima}\ni  \ldoublebracket g,g_1,t\rdoublebracket \longmapsto g H_{\ldoublebracket{g_1,t\rdoublebracket}} \in \Closed(G)
    \]
    restricts to   an isomorphism of topological groupoids from  the reduction of the   Oshima groupoid to $ \overline{\Oshima_+} \subseteq \Oshima$ to the Satake  groupoid $\Groupoid_{\Satake}$.
\end{theorem}

\begin{proof}
    A straightforward check shows that the map 
    \begin{align*}
        G \ltimes \overline{\Oshima_+}&\longrightarrow  \Groupoid_\Satake  \\
      (\ldoublebracket g_2,t_2\rdoublebracket, g, \ldoublebracket g_1,t_1 \rdoublebracket) &
      \longmapsto g H_{\ldoublebracket g_1,t_1 \rdoublebracket}
    \end{align*}
    is a morphism of groupoids; it is the homeomorphism of Theorem~\ref{thm-satake-in-oshima} on objects.  Moreover, the preimage of the object space of $\Satake$ of $\Groupoid_\Satake$ is precisely the reduction of the normal subgroupoid of Lemma \ref{lem-normal-subgroupoid-of-transformation-groupoid} to $\overline{M_+}$.  From this it follows easily that this map descends to a bijection from the reduction of $\Groupoid_\Oshima$ to $\overline{\Oshima_+}$ onto $\Groupoid_\Satake$.  
    Lemma \ref{lem-coset-groupoid-is-closed-subset} shows that we have convergence
    \[
      \gamma_iH_{\ldoublebracket g_i,t_i \rdoublebracket} \to 
      \gamma H_{\ldoublebracket g,t \rdoublebracket}
    \]
    in the Satake groupoid if and only if we have 
    \begin{equation*}
    H_{\ldoublebracket g_i,t_i \rdoublebracket} \to H_{\ldoublebracket g,t \rdoublebracket}, \qquad \text{and} \qquad
    \gamma_ih_i \to \gamma h, \quad  \exists h_i\in H_{\ldoublebracket g_i,t_i \rdoublebracket},\ \exists h\in H_{\ldoublebracket g,t \rdoublebracket}.
    \end{equation*}
    By Theorem \ref{thm-satake-in-oshima}, this is equivalent to
    \[
      \ldoublebracket \gamma_i,g_i,t_i \rdoublebracket 
      \to \ldoublebracket \gamma,g,t \rdoublebracket
    \]
    in the Oshima groupoid $\Groupoid_\Oshima$.
    The result follows.
\end{proof}

\section{The b-groupoid}
\label{sec:manifolds-with-corners}

In this section we shall identify the Oshima   groupoid with another Lie groupoid that is constructed using only geometric---rather than Lie-theoretic---features of the Oshima  space.

\subsection{Local coordinates on the Oshima groupoid}
We proved in Lemma~\ref{lem-big-cell}  that  the obvious map
\[
N \times \R^\Sigma   \longrightarrow  \bigl \{\, \ldoublebracket n,t \rdoublebracket \in \Oshima : n\in N,\,\,\, t\in \R^\Sigma\,\bigr \} 
\]
is a diffeomorphism onto an open subset  $\BigCell\subseteq \Oshima$. We shall now describe the   reduced  groupoid  
\[
\Groupoid_{\Oshima}\vert_{\BigCell} = \{\, \gamma \in \Groupoid_{\Oshima} : \operatorname{source}(\gamma)\in \BigCell \,\,\,\& \,\,\, \operatorname{target}(\gamma)\in \BigCell \,\}.
\]
Form the smooth manifold 
\[
\mathcal{W} = \bigl \{ \, 
\bigl ((n_2,t_2),a,(n_1,t_1)\bigr ) \in N {\times}\R^\Sigma {\times} A {\times} N {\times} \R^\Sigma
: t_2 = a \cdot t_1
\,\bigr \} .
\]
This is a Lie groupoid with object space $N{\times } \R^\Sigma$, and the  groupoid operations
\[
\begin{aligned}
\operatorname{source} ((n_2,t_2),a,(n_1,t_1)\bigr ) 
    & =(n_1,t_1)
    \\
\operatorname{target} ((n_2,t_2),a,(n_1,t_1)\bigr )  
    & =(n_2,t_2)
    \\
((n_3,t_3),a_2,(n_2,t_2)\bigr )   \circ ((n_2,t_2),a_1,(n_1,t_1)\bigr )  
    & = \bigl ( (n_3,t_3),a_2a_1,(n_1,t_1)\bigr ) .
\end{aligned}
\]
It is the Cartesian product of the pair groupoid on $N$ with the transformation groupoid for the action   $A$ on $\R^\Sigma$ in Definition~\ref{def-a-action}.

\begin{theorem}
\label{thm-big-cell-in-the-groupoid}
The smooth map 
\[
\begin{gathered} 
\mathcal {W}   \longrightarrow  \Groupoid_{\Oshima}\vert_{\BigCell}
\\
\bigl ( (n_2,t_2),a,(n_1,t_1)\bigr )  \longmapsto \ldoublebracket n_2an_1^{-1}, n_1,t_1\rdoublebracket
\end{gathered}
\]
\textup{(}notation of Definition~\ref{def-notation-for-elements-of-oshima-groupoid}\textup{)} is an isomorphism of Lie groupoids.
\end{theorem} 

\begin{proof} 
Is straightforward to check that the map  is a morphism of Lie groupoids.  
To see that it is surjective, suppose that $\ldoublebracket g,g_1,t\rdoublebracket \in \Groupoid_{\Oshima}\vert_{\BigCell}$. Then 
\[
\ldoublebracket g_1,t\rdoublebracket \in \BigCell 
\quad \text{and} \quad 
\ldoublebracket gg_1,t\rdoublebracket \in \BigCell ,
\]
from which it follows that 
\[
(g_1,t) = (n_1h_1a_1, a_1^{-1}\cdot t_1)  
\quad \text{and} \quad 
(gg_1,t) = (n_2h_2a_2, a_2^{-1}\cdot t_2)
\]
for some $n_1,n_2\in N$, $h_1\in H_{t_1}$, $h_2\in H_{t_2}$, $a_1,a_2\in A$ and $t_1,t_2\in \R^\Sigma$ (the identities in the display are identities in $G {\times} \R^\Sigma$).  This implies that 
\[
\ldoublebracket g,g_1,t\rdoublebracket 
     = \ldoublebracket g , n_1h_1a_1, t\rdoublebracket
    = \ldoublebracket g, n_1,t_1\rdoublebracket
\]
and that $g n_1h_1a_1 = n_2h_2a_2$.  Write this as 
\[
g = n_2h_2a_2a_1^{-1}h_1^{-1}n_1^{-1}
= n_2a n_1^{-1}(n_1a^{-1}h_2a n_1^{-1})(n_1 h_1^{-1}n_1^{-1}),
\]
where $a=a_2a_1^{-1}$. The two terms in parentheses on the right-hand side of the display are elements of $H_{\ldoublebracket n_1,t_1\rdoublebracket}$, and so by definition of the Oshima groupoid,
\[
\ldoublebracket g,g_1,t\rdoublebracket 
     = \ldoublebracket n_2an_1^{-1} , n_1 , t_1\rdoublebracket,
\]
as required.

To prove that the map is injective, we start by noting that if  
\[
\ldoublebracket n_2an_1^{-1},n_1,t \rdoublebracket
=
\ldoublebracket n_2'a'n_1'{}^{-1},n_1',t' \rdoublebracket \in \Groupoid_{\Oshima}\vert_{\BigCell},
\]
then  $n_1{=}n_1'$ and $t{=}t'$  by Lemma~\ref{lem-big-cell}, so that 
\[
\ldoublebracket n_2an_1^{-1},n_1,t \rdoublebracket
=
\ldoublebracket n_2'a'n_1^{-1},n_1,t \rdoublebracket \in \Groupoid_{\Oshima}\vert_{\BigCell}. 
\]
It then follows from the definition of the Oshima groupoid that 
\[
n_2an_1^{-1} =  n_2'a'n_1^{-1} (n_1h_1n_1^{-1})
\]
for some $ h_1 \in H_{ t_1 }$.  This means that $n_2a  =  n_2'a' h_1$,
and from this and  Lemma~\ref{lem:NAH} it follows that $n_2 {=} n_2'$ and  $a{=}a'$.

It remains to prove that our smooth map is a diffeomorphism, and for this we proceed as in the proof of Lemma~\ref{lem-big-cell}.  The relevant commutative diagram here is 
\[
\xymatrix@C=150pt{
 N   {\times} A {\times} N {\times} A {\times} \R^\Sigma  
 \ar[d]  
 \ar[r]^-{(n_2,a_2,n_1,a_,t)  \mapsto   ([n_2a_2 ,t], [n_1a_1,t] )}&  \bigG/\bigH\times _{\R^\Sigma} \bigG/\bigH   \ar[d]
 \\
 \mathcal{W} \ar[r] & \Groupoid_{\Oshima} ,
}
\]
where the left and right vertical maps  are 
\begin{multline*}
(n_2,a_2,n_1,a_1,t)\mapsto 
 \bigl((n_2,a_2\cdot t),a_2a_1^{-1},(n_1,a_1\cdot t)\bigr)
\\
\text{and}\quad 
([g_2,t],[g_1,t])\mapsto \ldoublebracket g_2g_1^{-1}, g_1, t\rdoublebracket,
\end{multline*}
respectively, and  the bottom horizontal map is from the statement of our theorem.  All maps other than the last-mentioned are submersions onto their images, and the left vertical map is surjective, so the bottom horizontal map  is actually a submersion too. Counting dimensions, we find that it is a diffeomorphism.
\end{proof}

\subsection{Codimension-one submanifolds with normal crossings}
In this subsection only, $\Oshima$ will denote \emph{any} smooth  manifold without boundary.  The following definition is taken from \cite[Sec.\,2.3]{Monthubert03}, where the term \emph{d\'ecoupage} is used. 

\begin{definition} 
\label{def-normal-crossings}
A  \emph{normal crossing structure} on $\Man$  is a finite collection of closed embeddings $\ManS{\to} \Man$ of smooth, codimen\-sion-one submanifolds without boundary such  that around any point  $m\in \Oshima$ there are local coordinates $x_1,\dots, x_n$ centered at $m$ with respect to which each $\ManS$ passing through $m$ is one of the coordinate hyperplanes $\{ x_j = 0\}$, with distinct  $\ManS$ corresponding to distinct  $j$.
\end{definition}

Monthubert \cite{Monthubert03} and Nistor, Weinsten and Xu \cite{NistorWeinsteinXu99}, have associated to a normal crossing structure on $\Oshima$ a  Lie groupoid with object space $\Oshima$, or rather several groupoids that differ in minor respects (they all share the same Lie algebroid). We shall consider only the following special case, and present  the version of the construction that is best suited to our needs.

\begin{definition} 
Let $\Man$ be a smooth   manifold with normal crossing structure $\{\ManS{\to} \Man\}$. We shall say that the normal crossing structure is \emph{simple} if 
$\Oshima$ is connected, if each $\ManS$ is connected, and if, for every $k$, 
the mutual complement 
\[
\bigcap_{j=1}^k \,  \Oshima \setminus \ManS_j  
\]
of any set of  $k$ hypersurfaces in $\Oshima$ has precisely $2^k$ connected components.
\end{definition}

When   $\Man$ is a smooth, connected  manifold with simple normal crossing structure, we may associate to $\Oshima$ a   groupoid  $\Gamma (\Oshima)$  with object space $\Oshima$ as follows:
\begin{enumerate}[\rm (i)]

\item The morphism space is the set of triples $(m', T, m)$, where 
        \begin{enumerate}[\rm (a)]

        \item $m'$ and $m$ lie on the same collection $\ManS_1,\dots ,\ManS_k$ of hypersurfaces,

        \item $m'$ and $m$ lie on the same side  of all   the   hypersurfaces on which they do not lie, and 

        \item $T=\{ T_1,\dots , T_k\}$  is a family of orientation-preserving linear isomorphisms
        \[
        T_j\colon N_{m} \ManS_j\longrightarrow N_{m'} \ManS_j\qquad (j=1,\dots , k) 
        \]
        between fibers of the normal bundles of the hypersurfaces $\ManS_1,\dots ,\ManS_k$ in $\Man$.  (The hypothesis of simplicity implies that each normal bundle is orientable; moreover  if $T_j$ is orientation-preserving for one orientation, then it is orientation-preserving for the other, too.  So a choice of orientation is not required.) 
\end{enumerate} 

\item The source and target maps are 
        \[
        \operatorname{source}\colon (m', T, m) \mapsto m
        \quad \text{and} \quad 
        \operatorname{target}\colon (m', T, m) \mapsto m' ,
        \]
    and the composition law is 
        \[
        (m'', T', m')\circ (m', T, m) = (m'', T'T, m),
        \]
    where $T'T = \{\, T'_1T_1,\dots ,T'_kT_k\,\}$.
\end{enumerate}
See \cite[Sec.\,3]{Monthubert03}.

To understand how to equip the  groupoid $\Gamma ( \Oshima )$ with a Lie groupoid structure, let us consider the model case, in which $\Oshima = \R^{n}$, and in which the hypersurfaces $\ManS\to \Oshima$ are the hyperplanes $\{ x_j=0\}$ for $j=1,\dots, p$.   Denote by $A$ the abelian group $\R^p_+$, and let $A$ act on $\Oshima$ by coordinate-wise multiplication in the first $p$ places. Then form the transformation Lie groupoid 
\[
A \ltimes \Oshima =  \bigl\{ \, (m',a,m) \in \Oshima {\times} A {\times} \Oshima : 
m' = a \cdot m \, \bigr \} ,
\]
in which the  target  and source maps are the projections on the the first and third factors, and the composition law is 
\[
(m'',a',m') \circ (m',a,m) = (m'', a'a, m).
\]

In the model case, the normal bundle for each $\ManS{\to} \Oshima$ identifies with the constant bundle with fiber $\R$ in the obvious way, and so  if $(m',T,m)$ is a morphism in $\Gamma (\Oshima)$, then to each $T_j$ there corresponds to a positive real number. In the model case there is therefore an isomorphism 
\[
\begin{gathered}
\Gamma(\Oshima)  \stackrel\cong \longrightarrow A \ltimes \Oshima
\\
(m',T,m)  \longmapsto (m', a(T), m) ,
\end{gathered}
\]
 where 
\[
a(T)_j = \begin{cases} 
T_j & m_j =0 
\\
m'_j/m_j & m_j \ne 0.
\end{cases}
\]
The isomorphism gives $\Gamma (\Oshima)$ a Lie groupoid structure in the model case.  The general case is handled similarly, using local coordinates adapted to the embeddings $\ManS{\to}\Oshima$, as in Definition~\ref{def-normal-crossings}.

\begin{remark}
We might call the Lie groupoid $\Gamma (\Man)$ above  the \emph{$b$-groupoid} associated to the given simple normal crossing structure on  $\Man$, because it corresponds (up to some minor variations concerning connectedness of source fibers) to the groupoid introduced by Monthubert \cite{Monthubert:b-groupoid} and Nistor, Weinstein and Xu \cite{NistorWeinsteinXu99} in the study of Melrose's $b$-calculus \cite{Melrose:APS}.
The $C^\infty (M)$-module of vector fields on $\Man$ that are tangent to each of the embeddings  $\ManS{\to} \Man$ is finitely generated and projective, and is therefore the module of smooth sections of a vector bundle over $\Man$,  called the \emph{$b$-tangent bundle} $T^bM$.  It is equipped with a Lie bracket structure on its smooth sections and a natural morphism $T^b \Man\to TM$, and these together give $T^b\Man$ the structure of a Lie algebroid.  This is the Lie algebroid of $\Gamma (\Oshima)$.
\end{remark}

\subsection{Normal crossing structure on the Oshima space}
We shall now return to  the Oshima space    $\Oshima$ associated to a real reductive group, as in Section~\ref{sec-oshima-space}.  For each $\alpha\in \Sigma$ we may define a subset  $\mathcal{S}_\alpha\subseteq \Oshima$ by 
\[
S_\alpha = \{\, \ldoublebracket g,t\rdoublebracket  \in \Oshima : t_\alpha =0 \,\} .
\]
\begin{lemma} Each $\mathcal{S}_\alpha$ is a smooth, closed, codimension one submanifold of $\Oshima$, and the collection of closed embeddings $\ManS_\alpha{\to}\Oshima$ \textup{(}$\alpha \in \Sigma$\textup{)} gives $\Oshima$ a simple normal crossing structure.
\end{lemma} 

\begin{proof} 
This is made evident by  the coordinatization in Lemma~\ref{lem-big-cell}. 
\end{proof}

\begin{remark}
    \label{rem-finite-group-action}
Another way to see the normal crossing structure is to observe that 
the finite group $\Z_2^\Sigma$ acts on $\R^\Sigma$ in the following way: think of $\Z_2$ as the multiplicative group $\{ \pm 1\}$, and then define 
\[
(s \cdot t ) _\alpha = s_\alpha t_\alpha\qquad \forall \alpha \in \Sigma
\] 
The $\Z_2^\Sigma$ action on $\R^\Sigma$ passes to compatible actions on $\bigG$, $\bigG/\bigH$ and $\Oshima$, all commuting with the left action of $G$.  For example, the action on $\Oshima$ is 
\[
 s\cdot \ldoublebracket g,t\rdoublebracket = \ldoublebracket g,st\rdoublebracket\qquad \forall s\in \Z_2 ^\Sigma,\,\,\, \forall  \ldoublebracket g,t\rdoublebracket \in \Oshima.
\]
The fixed point set in $\Oshima$ associated to any element of $\Z_2^\Sigma$ is automatically a smooth submanifold.  The fixed point sets for the standard generators in $\Z_2^\Sigma$ are our hypersurface embeddings $\ManS_\alpha {\to}\Oshima$.
\end{remark}

\subsection{Identification of the Oshima groupoid with the  b-groupoid}
The action of $G$  on the Oshima space  $\Man$ maps each $\ManS_\alpha$ into iteslf, and it also maps each component of each $\Oshima \setminus \ManS_\alpha$ into itself. As a result the induced action on the normal bundle of $\ManS_\alpha{\to}\Oshima$ is orientation-preserving, and we may define a morphism of Lie groupoids 
\begin{equation}
    \label{eq-morphism-to-b-groupoid}
\begin{gathered}
G\ltimes \Oshima \longrightarrow \Gamma (\Oshima) 
\\
(m',\gamma ,m) \longmapsto (m', [\gamma], m),
\end{gathered}
\end{equation}
where $[\gamma]$ indicates the induced action of $\gamma \in G$ on normal bundles  (from the fibers at $m$ to the fibers at $m'$).

\begin{theorem} 
\label{thm-isomorphism-with-b-groupoid}
The map \eqref{eq-morphism-to-b-groupoid} induces an isomorphism of Lie groupoids from $\Groupoid_{\Oshima}$ \textup{(}viewed as a quotient of the action groupoid\textup{)} to the Lie groupoid  $\Gamma (\Oshima)$.
\end{theorem}

Our proof of Theorem~\ref{thm-isomorphism-with-b-groupoid} will  use the following technical computation.

\begin{lemma}
    \label{lem-zero-derivative-extension}
    Let $t_0 \in \R^\Sigma$ and let $h\in H_{t_0}$.  There is a smooth section of $\bigH$, which we shall write as $t\mapsto (h_t,t)$ such that $h_{t_0}=h$ and 
    \[
    \left . \frac{\partial}{\partial t_\alpha} \right  \vert _{t=t_0} h_{t} = 0 
    \]
    for every $\alpha\in \Sigma$ such that $t_{0,\alpha } = 0$.
\end{lemma}

\begin{proof}
We may write 
\[
h_{t_0} = k \exp(Y_1) \cdots \exp(Y_k)
\]
for some $k\in M$ and $Y_1,\dots , Y_k\in \mathfrak{h}_{t_0}$. We define   
\[
h_t = k \exp(Y_{1,t}) \cdots \exp(Y_{k,t}),
\]
where each $Y_{j,t}$ has the form
\[
Y_{j,t} = Z + \sum _{\gamma \in \Delta ^+(\mathfrak{g},\mathfrak{a})}
       t^{2 \gamma}X_\gamma  {+} \theta(X_\gamma )   
\]
for some  $Z\in \mathfrak{m}$, $X_{\gamma} \in \mathfrak{g}_\gamma$,   and where $Y_{j,t_0}=Y_j$.  Compare Definition~\ref{def-h-t}.   Since $t^{2\gamma}$ vanishes to second order in $t_\alpha$ when $t_\alpha=0$, 
the map $t\mapsto h_t$ has the required properties.
\end{proof}

\begin{lemma}
    \label{lem-h-elements-act-trivially-on-normal-bundles}
    Let $t_0 \in \R^\Sigma$ and suppose that $t_{0,\alpha}=0$ for some $\alpha \in \Sigma$.  Every element of  $H_{t_0}$, viewed as a diffeomorphism from $\Oshima$ to itself using the action of $G$ on $\Oshima$, acts as the identity operator on the fiber of the normal bundle of the hypersurface $\ManS_\alpha$ at the point $\ldoublebracket e,t_0\rdoublebracket\in \Oshima$.
\end{lemma}

\begin{proof}
    The fiber of the normal bundle of the hypersurface $\ManS_\alpha$ at the point $\ldoublebracket e,t_0\rdoublebracket\in \Oshima$ is generated by the tangent vector 
    \[
    \left .\frac{\partial}{\partial t_{\alpha}} \right |_{t_0} \ldoublebracket e, t\rdoublebracket  \in T _{\ldoublebracket e,t_0\rdoublebracket} \Oshima.
    \]
    Given $h\in H_{t_0}$, if $h_t$ is as in the previous lemma, with $h_{t_0} = h$, then by the Leibniz rule and the lemma,
    \[
    \begin{aligned}
        \left .\frac{\partial}{\partial t_{\alpha}} \right |_{t_0} \ldoublebracket e, t\rdoublebracket  
        &= \left .\frac{\partial}{\partial t_{\alpha}} \right |_{t_0} \ldoublebracket h_t, t\rdoublebracket  
        = \left .\frac{\partial}{\partial t_{\alpha}} \right |_{t_0} \ldoublebracket h, t\rdoublebracket  + \left .\frac{\partial}{\partial t_{\alpha}} \right |_{t_0} \ldoublebracket h_t, t_0\rdoublebracket  
         = \left .\frac{\partial}{\partial t_{\alpha}} \right |_{t_0} h\cdot \ldoublebracket e, t\rdoublebracket,
    \end{aligned}    
    \]
    which proves the lemma.
\end{proof}

\begin{proof}[Proof of Theorem~\ref{thm-isomorphism-with-b-groupoid}]
Recall that the Oshima groupoid may be obtained as the quotient of the transformation groupoid $G \ltimes \Oshima$ by the smooth family of groups $H_{\ldoublebracket g,t\rdoublebracket}$ over $\Oshima$, viewed as a normal subgroupoid of the transformation groupoid. It follows from Lemma~\ref{lem-h-elements-act-trivially-on-normal-bundles} that if $h\in H_{\ldoublebracket g,t\rdoublebracket}$, and if $t_\alpha = 0$, then $h$ acts as the identity operator on the fiber of the normal bundle of the hypersurface $\ManS_\alpha$ at the point $\ldoublebracket e,t_0\rdoublebracket\in \Oshima$.  So the morphism of Lie groupoids
\[
G\ltimes \Oshima \longrightarrow \Gamma (\Oshima) 
\]
in \eqref{eq-morphism-to-b-groupoid} factors through a morphism of Lie groupoids
\begin{equation}
    \label{eq-functor-from-oshima-groupoid-to-b-groupoid}
\Groupoid_{\Oshima} \longrightarrow \Gamma (\Oshima) 
\end{equation}
It remains to prove that the latter is a diffeomorphism.

For this, we will compose \eqref{eq-functor-from-oshima-groupoid-to-b-groupoid} with two local diffeomorphisms.  Firstly, we have the Lie groupoid isomorphism
\begin{equation}
    \label{eq-functor-from-w-to-oshima-groupoid}
  \mathcal{W} \stackrel\cong \longrightarrow    \Groupoid_{\Oshima}\vert_{\BigCell} 
\end{equation}
from Theorem~\ref{thm-big-cell-in-the-groupoid}.  Secondly, the diffeomorphism of Lemma \ref{lem-big-cell} gives a Lie groupoid isomorphism 
\begin{equation}
\label{eq-functor-from-b-groupoid-to-big-cell-b-groupoid}
    \Gamma(\Oshima)\vert_{\BigCell} \stackrel\cong\longrightarrow \Gamma(N\times\R^\Sigma),
\end{equation}
where the right-hand side is the $b$-groupoid of $N\times\R^\Sigma$ with simple normal crossing structure coming from the coordinate hyperplanes in $\R^\Sigma$.  Composing \eqref{eq-functor-from-w-to-oshima-groupoid}, \eqref{eq-functor-from-oshima-groupoid-to-b-groupoid} and \eqref{eq-functor-from-b-groupoid-to-big-cell-b-groupoid}, we get the morphism of Lie groupoids $\mathcal{W}   \to \Gamma (N \times \R^\Sigma )$ given by
\[
\bigl ((n_2,t_2), a, (n_1, t_1)\bigr ) 
\longmapsto 
\bigl (  ( n_2,t_2) , [n_2an_1^{-1} ] , (n_1,t_1) \bigr) .
\]
Since it is evident that $n_2$ and $n_1$ act trivially on normal bundles for $N{\times}\R^\Sigma$, we may also write this as 
\[
\bigl ((n_2,t_2), a, (n_1, t_1)\bigr ) 
\longmapsto 
\bigl (  ( n_2,t_2) , [a  ] , (n_1,t_1) \bigr) ,
\]
which is an isomorphism of Lie groupoids.  

This proves that
\[
   \Groupoid_\Oshima\vert_\BigCell \cong \Gamma(\Oshima)\vert_{\BigCell}.
\]
To complete the proof, note that the morphism of Lie groupoids \eqref{eq-functor-from-oshima-groupoid-to-b-groupoid} is equivariant with respect to the $G$-actions
\[
  g\cdot (m',\gamma, m) = (g\cdot m', g\gamma g^{-1}, g\cdot m), 
  \qquad \bigl(g\in G,\ (m',\gamma,m) \in \Groupoid_\Oshima \bigr),
\]
and
\[
  g\cdot (m',T, m) = (g\cdot m', [g]T[g^{-1}], m), 
  \qquad \bigl(g\in G,\ (m',T,m) \in \Groupoid_\Oshima \bigr).
\]
Since the $G$-translates of $\BigCell$ cover $\Oshima$, the result follows.
\end{proof}

\bibliographystyle{alpha}
\bibliography{refs}

\end{document}